\documentclass[12pt]{amsart}

\usepackage[a4paper]{geometry}
\usepackage{xcolor}

\usepackage[T1]{fontenc}
\usepackage[latin9]{inputenc}
\usepackage[english]{babel}
\usepackage{amssymb}
\usepackage{mathtools}
\usepackage[tt=false]{libertine}
\usepackage[varbb]{newpxmath}
\usepackage{bbm}
\usepackage[colorlinks=true,linkcolor=blue,citecolor=blue,urlcolor=blue]{hyperref}
\usepackage[capitalize,nameinlink]{cleveref}

\DeclareMathOperator*{\esssup}{ess\,sup}

\renewcommand\leq\leqslant
\renewcommand\geq\geqslant
\newcommand\eps\varepsilon

\newcommand{\E}{\mathsf E}

\newcommand{\R}{\mathbb R}

\newcommand{\F}{\mathcal F}

\def\PP{{\rm P}}

\numberwithin{equation}{section}
\numberwithin{figure}{section}

\theoremstyle{plain}
\newtheorem{thm}{Theorem}
\newtheorem{lem}{Lemma}

\theoremstyle{definition}
\newtheorem{assumption}{Assumption}

\theoremstyle{remark}
\newtheorem{rem}{Remark}

\title[Ergodic Properties of Non-Linear Density-Dependent Perturbations of the OU Process]{Ergodic Properties of Non-Linear Density-Dependent Perturbations of the Ornstein--Uhlenbeck Process}
\author{Denis Belomestny}
\address{Duisburg-Essen University, Germany}
\email{denis.belomestny@uni-due.de}
\thanks{The present research is supported by the Deutsche Forschungsgemeinschaft through the
BE 3961/7-1 ``Statistische Inferenz f\"ur
Teilchensysteme und McKean-Vlasov-SDEs mit singul\"aren Kernen''}
\author{Ekaterina Morozova}
\address{Duisburg-Essen University, Germany}
\email{ekaterina.morozova@uni-due.de}
\thanks{The second author is the corresponding author.}
\date{}

\begin{document}
\begin{abstract}
    The present paper considers McKean-Vlasov SDEs with density-dependent spatially unbounded drift, which may be viewed as a non-linear density-dependent perturbation of the Ornstein-Uhlenbeck process. We develop a comprehensive theoretical framework for this class of equations. First, we establish strong well-posedness and derive optimal Gaussian pointwise bounds for both the solution density and its gradient. Then we derive an explicit expression for the stationary density and show that it satisfies logarithmic Sobolev and Poincar\'e inequalities. Finally, we prove exponential convergence to equilibrium in the  \(\chi^2\)-metric.
\end{abstract}
\maketitle

\section{Introduction}
\label{sec:intro}

McKean--Vlasov stochastic differential equations (MVSDEs), also referred to as dist-ribution-dependent stochastic differential equations, have become a cornerstone of modern probability theory and stochastic analysis. 
Originally introduced to describe the limiting macroscopic behavior of large systems of weakly interacting particles in statistical mechanics, MVSDEs are now ubiquitous across various fields, including mathematical finance, population dynamics, and mean-field games. 
In the classical setting, the drift and diffusion coefficients depend on the marginal law of the solution through global integral operators (such as convolution with a regular interaction kernel), reflecting a mean-field interaction where each particle is influenced by the averaged state of the entire population.
\par
In this paper, we consider a specific and highly non-linear class of MVSDEs on $\mathbb{R}^{d}$, $d \ge 1$, which features \emph{Nemytskii-type} coefficients. In such models, the interaction is strictly local: the coefficients depend directly on the pointwise value of the probability density function of the state. Specifically, we study the dynamics of a process $X = (X_t)_{t\ge0}$ governed by the equation
\begin{equation}\label{eq:MV-SDE}
dX_t = -\Xi(p_t(X_t))\nabla\Phi(X_t)\,dt
+ \sqrt{2}\,dW_t, \quad X_0 \sim p_0,
\end{equation}
where $(W_{t})_{t\ge0}$ represents a standard $d$-dimensional Brownian motion on a filtered probability space $(\Omega, \mathcal{F},(\mathcal{F}_{t}), \mathsf{P})$, and $p_t$ denotes the probability density function of $X_t$ with respect to the Lebesgue measure. The function $\Phi: \mathbb{R}^d \to \mathbb{R}$ acts as a spatial confining potential, while the scalar function $\Xi: \mathbb{R}_+ \to \mathbb{R}_+$ modulates the strength of this potential based on the local spatial concentration of the mass.
\par
By It\^o's formula, the marginal densities $p_t(x)$ of a weak solution to \eqref{eq:MV-SDE} are intimately linked to the non-linear Fokker--Planck equation (NFPE):
\begin{equation}\label{eq:NFPE}
\partial_{t}p - \Delta p - \operatorname{div}\big(\nabla\Phi(x)\Xi(p)p\big) = 0, \quad p(0,x) = p_{0}(x),
\end{equation}
where $p_{0}$ is the initial density of $X_{0}$. Conversely, if $p$ serves as a sufficiently regular distributional solution to the NFPE \eqref{eq:NFPE}, it provides the marginal distributions for a weak solution of the MVSDE \eqref{eq:MV-SDE} on a corresponding probability space such that $p_t(x)dx = \mathsf{P} \circ (X_t)^{-1}(dx)$. 

Throughout this paper, we work under the following structural assumptions on the coefficients and the initial condition.

\begin{assumption}
\label{ass:main}
    Assume that:
    \begin{enumerate}
    \item \textit{$\Phi(x) = 1 + \|x\|_{2}^{2}$;}
    \item \textit{$\Xi \in C_{b}(\mathbb{R}_{+}) \cap C^{1}(\mathbb{R}_{+})$ is a globally Lipschitz function such that $0 < \kappa \le \Xi(r) \le K$ for all $r \ge 0$;}
    \item \textit{$p_{0} \in L^{\infty}(\mathbb{R}^{d}) \cap L^{1}(\mathbb{R}^{d})$ is such that $p_{0}(x) \le c_{0}\exp\{-b_{0}\|x\|_{2}^{2}\}$ with some $c_{0}, b_{0} > 0$.}
\end{enumerate}
\end{assumption} 

The first condition implies that $\nabla\Phi(x) = 2x$, which allows us to view \eqref{eq:MV-SDE} as a density-dependent, non-linear perturbation of the classical Ornstein--Uhlenbeck (OU) process. Surprisingly, even though the linear Ornstein--Uhlenbeck process is exquisitely well-studied, introducing a local Nemytskii-type coefficient combined with a spatial drift of linear growth presents profound analytical difficulties.

\subsection{Literature Overview and Motivation}
The well-posedness and asymptotic behavior of classical MVSDEs---where the coefficients depend continuously on the probability measure equipped with the Wasserstein metric---are mathematically well established. In such regimes, one can apply such powerful probabilistic tools as synchronous couplings and Wasserstein contraction arguments. Consequently, the literature readily accommodates systems featuring spatial drifts that are unbounded, covering cases of linear growth \cite{Mishura2020}, as well as singular and superlinear drifts \cite{Kumar2022, Ren2023, Rockner2021}. Related one-dimensional nonlocal models, where the drift depends on the cumulative distribution function of the law, were studied by Jourdain and Malrieu via propagation of chaos and Poincar\'e inequalities \cite{jourdain2008propagation}.
However, the literature concerning MVSDEs with Nemytskii-type coefficients is significantly more scarce. Because Nemytskii operators rely on the pointwise evaluation of the density, the mapping $\mu \mapsto \Xi(d\mu/dx(\cdot))$ fundamentally lacks continuity with respect to standard weak topologies or the Wasserstein distance. A small perturbation in Wasserstein space provides no pointwise control over the corresponding densities, causing traditional probabilistic fixed-point arguments formulated on the space of probability measures to break down. To overcome this, the study of local density-dependent equations has required heavily analytic approaches, relying on parabolic PDE theory, Sobolev estimates, or weighted $L^1$-semigroups. The strong well-posedness of time-dependent Nemytskii equations has been studied by Grube \cite{Grube2024}, and a weighted $L^1$-semigroup approach for related NFPEs was recently developed by Rehmeier \cite{Rehmeier2023}. The comprehensive monograph by Barbu and R\"ockner \cite{Barbu2024} establishes generalized porous media and NFPE flows that offer considerable flexibility. We also refer to the recent review by Bondi et al.\ \cite{Bondi2025} for a broader summary of the state-of-the-art on both distribution- and density-dependent singular equations.
Yet, a common limitation in much of the existing literature on Nemytskii-type distribution-dependent equations is the reliance on spatially bounded drift coefficients, or on assumptions that effectively prevent the unbounded spatial part from interacting too strongly with the local density dependence. Such restrictions are technically important, since bounded spatial drifts help preclude mass escape and provide compactness and regularity estimates for the associated densities. In the present model, however, \(\nabla\Phi(x)=2x\), so the drift in \eqref{eq:MV-SDE} is genuinely unbounded, growing linearly at infinity, and is coupled multiplicatively with the non-linear Nemytskii coefficient.

We note that, at the PDE level, unbounded spatial drifts in NFPEs with local and nonlocal dependence have been treated in considerable generality by Bogachev et al.~\cite{BSS24}. Their framework allows coefficients with polynomial growth in the spatial variable and includes local dependence on the density. The results there are formulated for prescribed total mass and, in the parts most relevant to our setting, involve a coefficient-dependent small-mass threshold \(M\in(0,M_0)\). In contrast, the present research is tailored to the probability McKean--Vlasov setting, where \(M=1\) is fixed, and aims at more model-specific and quantitative conclusions: we identify the associated McKean--Vlasov dynamics, characterize the invariant probability density explicitly, establish functional inequalities for it, and prove exponential convergence in \(\chi^2\)-divergence. This makes the rigorous study of \eqref{eq:MV-SDE}, particularly its long-time ergodicity, a problem of significant mathematical interest.


\subsection{Methodology and Main Contributions}
The primary objective of this paper is to bridge the theoretical gap outlined above by establishing the global well-posedness, exact stationary characteristics, and exponential ergodicity for the unbounded density-dependent system \eqref{eq:MV-SDE}. Our main contributions can be summarized as follows:

\begin{enumerate}
    \item \textit{Strong Well-Posedness and Regularity:} We first establish the existence and uniqueness of strong solutions to the Nemytskii MVSDE \eqref{eq:MV-SDE} (Theorem~\ref{thm:wellposednessMV}). By formulating a fixed-point argument in a weighted, exponentially-decaying \(L^2\)-space and employing a Zvonkin-type space-time transformation, we circumvent the unbounded drift and construct a unique locally classical solution \(p_t(x)\) for the NFPE~\eqref{eq:NFPE} that inherently satisfies baseline Gaussian tail bounds.
    
    \item \textit{Stationary Distribution and Functional Inequalities:} We explicitly characterize the unique, strictly positive stationary probability density \(\pi(x)\) of the system (Lemma~\ref{lem:stationary}). Because the model behaves asymptotically as a bounded perturbation of a Gaussian potential, we rigorously show via the Holley--Stroock perturbation lemma (see, e.g., \cite{Cattiaux2022}) that \(\pi\) retains Gaussian tail behavior and strictly satisfies both the Log-Sobolev and Poincar\'e inequalities (Lemma~\ref{LSI}).
    
    \item \textit{Optimal Pointwise Bounds:} To study the long-time exponential ergodicity of the system in the $\chi^2$-divergence, the mere existence of Gaussian tails is insufficient. A rigorous justification requires uniform-in-time, rather sharp exponential bounds for both the density $p_t(x)$ and its spatial gradient $\nabla_x p_t(x)$. We first enhance the initial bound for the density \(p_t\) obtained in the fixed-point argument, establishing a Gaussian pointwise bound with uniform constants and an optimal rate \(\Xi(0)\) (Lemma~\ref{lem:unifbound}). Afterwards this result is used to enhance the baseline bound for the gradient constructed in Theorem~\ref{thm:grad_pt-bound}. To achieve this, we pivot to an analytic PDE approach. By representing the NFPE as an additive, continuous perturbation of the Ornstein--Uhlenbeck generator, we leverage Duhamel's variation of constants formula along the exact OU heat kernel. Using this strategy, we again establish an upper pointwise Gaussian bound with rate \(\Xi(0)\), so both bounds for the density~\eqref{gaussbound_pt_unif} and its gradient~\eqref{grad_est} match the decay rate of the unperturbed OU process.

    \item \textit{Exponential Ergodicity in \(\chi^2\)-Divergence:} Given these exact, optimized pointwise estimates, we justify the generalised integration-by-parts argument required to differentiate the \(\chi^2\)-divergence \(\chi^2(p_t \,\|\, \pi)\). Under a mild control on the Lipschitz constant of the interaction function \(\Xi\), we prove that the law of the process converges exponentially fast to its unique stationary state (Theorem~\ref{thm:chisq}).
\end{enumerate}

\subsection{Organization of the Paper}
The paper is structured as follows. In Section~\ref{sec:wellposedness}, we formally state and prove the well-posedness of the NFPE and the strong uniqueness of the corresponding MVSDE, establishing a uniform optimal bound for the density and an initial baseline bound for its gradient. Section~\ref{sec:stationary} is devoted to characterizing the exact stationary distribution and establishing the associated Poincar\'e inequality. In Section 3, we state the precise optimal spatial bounds for the gradient of the density, and prove the exponential ergodicity in the $\chi^2$-metric. Section~\ref{sec:chisq} establishes uniform-in-time polynomial moment bounds for the underlying stochastic process, as well as the moment bounds for the exponential of its integrated norm. 
Section~\ref{sec:proofs} contains the proofs of the established results. Finally, Section~\ref{sec:aux} gathers some auxiliary results necessary for the proofs.

\textbf{Notation.} Throughout the paper, \(\Vert \cdot\Vert_2\) denotes the standard Euclidean norm on \(\mathbb{R}^d\), and \(C^{\beta/2,\beta}((0,T)\times \R^d)\) for \(T, \beta>0\) stands for the parabolic H\"older space.
The sign \(\lesssim\) (resp.\ \(\gtrsim\)) denotes inequality up to some constants independent of the unknown distribution.

\section{Well-posedness and basic bounds}\label{sec:wellposedness}

We start with establishing the well-posedness of~\eqref{eq:MV-SDE}, as well as some basic bound for the corresponding density. 
\begin{thm}\label{thm:wellposednessMV}
    Suppose that, for a given time horizon \(T>0\), Assumption~\ref{ass:main} holds with \(b_0>\kappa/2\) and \(0<\kappa<K<\infty\) such that \(K/\kappa\leq 2\). Then the SDE~\eqref{eq:MV-SDE} has a unique strong solution. Moreover, this solution has a density \(p\in C^{\beta/2,\beta}_{\mathrm{loc}}((0,T)\times\R^d)\), \(\beta\in(0,1)\), satisfying~\eqref{eq:NFPE} and such that
    \begin{equation}\label{gaussbound_pt}
    p_t(x)\leq c_T\exp\{-\gamma_T\Vert x\Vert_2^2\},
    \quad \forall\,(t,x)\in(0,T)\times\R^d,
    \end{equation}
    with some \(c_T>0\) and \(\gamma_T\in(0,\kappa/2)\). 
\end{thm}

While the bound~\eqref{gaussbound_pt} already characterises the behaviour of \(p_t\), the constants \(c_T\) and \(\gamma_T\) can, in general, depend on the terminal time \(T\). Under a slightly stronger assumption on the initial condition, one can achieve a bound with uniform constants, as well as enhance the rate of convergence in the exponential. The precise statement is given by the following lemma. 

\begin{lem}\label{lem:unifbound}
Under the same assumptions as in Theorem~\ref{thm:wellposednessMV}, assuming also \(b_0>\Xi(0)\), it holds that, moreover, 
\begin{equation}\label{gaussbound_pt_unif}
    p_t(x)\leq c\exp\{-\Xi(0)\Vert x\Vert_2^2\},
    \quad \forall\,(t,x)\in(0,T)\times\R^d,
    \end{equation}
    with some constant \(c:=c(b_0, c_0, \Xi(0))>0\) independent of \(T\). 	
\end{lem}

Finally, by a rather standard parabolic bootstrap argument, one can extend \(p_t\) to a locally classical solution to~\eqref{eq:NFPE}, and it can furthermore be shown that the gradient \(\nabla p_t\) also satisfies a Gaussian upper bound.

\begin{thm}\label{C2+grad}
    Under the same assumptions as in Lemma~\ref{lem:unifbound}, the weak solution \(p\) of~\eqref{eq:NFPE} belongs to the class \(C^{1+\beta/2,2+\beta}_{\mathrm{loc}}\bigl((0,T)\times\mathbb R^d\bigr)\) for some \(\beta\in(0,1)\) and any \(T>0\) fixed. In particular, it is a locally classical solution. Moreover, for any \(t\in(0,T)\) and \(x\in\R^d\) its gradient satisfies
    \begin{equation}\label{gaussbound_grad}
    \Vert \nabla p_t(x)\Vert_2
    \leq \frac{c_1(d,K,b_0,c_0,\Xi(0),L_{\mathrm{Lip}}(\Xi))}{\min\left(1,\sqrt{t/2}\right)}
    (1+\Vert x\Vert_2)\exp\left\{
    -\frac{\Xi(0)}{4}\Vert x\Vert_2^2
    \right\},
    \end{equation}
    where \(L_{\mathrm{Lip}}(\Xi)\) is the Lipschitz constant of \(\Xi\), and \(c_1(d,K,b_0,c_0,\Xi(0),L_{\mathrm{Lip}}(\Xi))\) is some constant. 
\end{thm}

\section{Stationary distribution}\label{sec:stationary}

The following lemma characterises the form of the stationary distribution of~\eqref{eq:MV-SDE}. It is worth mentioning that the form of the stationary density coincides with that obtained in Chapter 4 of~\cite{barbu2024nonlinear}.
\begin{lem}\label{lem:stationary}
Under the same assumptions as in Lemma~\ref{lem:unifbound}, there exists a unique probability density \(\pi\) solving the stationary form of~\eqref{eq:NFPE}. Moreover, \(\pi\in C^1(\R^d)\) and satisfies 
\begin{equation}\label{pi_general}
\pi(x)=g^{-1}(-\Phi(x)+\mu),\quad x\in\mathbb{R}^{d}
\end{equation}
where 
\begin{eqnarray}\label{g}
g(r)=\int_{1}^{r}\frac{1}{s\,\Xi(s)}\,ds,\quad r> 0,
\end{eqnarray}
and $\mu$ is the unique number in $\mathbb{R}$ such that 
\begin{equation}\label{gnorm}
\int_{\mathbb{R}^{d}}g^{-1}(-\Phi(x)+\mu)\,dx=1.
\end{equation}
\end{lem}

As was mentioned earlier, the model~\eqref{eq:MV-SDE} is essentially a non-linear perturbation of a classical Ornstein-Uhlenbeck process, for which the stationary measure is Gaussian and therefore log-concave. While the stationary density \(\pi\) need not be strictly log-concave in general, one can show that, being a bounded perturbation of the Gaussian density, it enjoys some of the similar properties, such as exponential bounds and the Log-Sobolev (and therefore Poincar\'e) inequality. 
\begin{lem}\label{LSI}
    The stationary density \(\pi\) satisfies 
    \begin{equation}\label{pi_bound}
   \pi(0)e^{-K\Vert x\Vert_2^2}
   \leq \pi(x)
   \leq \pi(0)e^{-\varkappa\Vert x\Vert_2^2},
   \quad \forall\,x\in\R^d.
   \end{equation}
   In addition, \(\pi\) satisfies the Log-Sobolev and Poincar\'e inequalities with constants \(C_{\mathrm{LSI}}\) resp.\ \(C_\mathrm{Poin}\) such that
   \begin{equation}\label{Poin-const}
   C_\mathrm{Poin}
   \leq C_{\mathrm{LSI}}
   \leq \frac{1}{2\Xi(0)}
   \exp\left\{
   \frac{L_{\mathrm{Lip}}(\Xi)\pi(0)}{\kappa}
   \right\}.
   \end{equation}
\end{lem}

\begin{rem}
    It can be seen that the bounds~\eqref{pi_bound} and~\eqref{Poin-const} formally depend on \(\pi(0)\), i.e., the unknown density evaluated at zero. Using the already constructed bound~\eqref{pi_bound}, one can remove this undesired dependency. Indeed, since
    \[
    1=\int_{\R^d} \pi(x)\,dx 
    \leq \pi(0)\int_{\R^d} e^{-\kappa\Vert x\Vert_2^2}\,dx
    =\pi(0)\left(
    \frac{\pi}{\kappa}
    \right)^{d/2},
    \]
    it holds that \(\pi(0)\geq (\kappa/\pi)^{d/2}\). Similarly,
    \[
    1=\int_{\R^d}\pi(x)\,dx
    \geq \pi(0)\int_{\R^d} e^{-K\Vert x\Vert_2^2}\,dx
    \geq \pi(0)
    \left(
    \frac{\pi}{K}
    \right)^{d/2},
    \]
    which yields \(\pi(0)\leq (K/\pi)^{d/2}\). Hence, one obtains slighly looser bounds
    \[
    \left(
    \frac{\kappa}{\pi}
    \right)^{d/2}
    e^{-K\Vert x\Vert_2^2}
    \leq \pi(x)
    \leq \left(
    \frac{K}{\pi}
    \right)^{d/2}
    e^{-\kappa\Vert x\Vert_2^2}
    \]
    and
    \begin{equation}\label{C-Poin2}
    C_{\mathrm{Poin}}
    \leq C_{\mathrm{LSI}}
    \leq \frac{1}{2\Xi(0)}
    \exp\left\{
    \frac{L_{\mathrm{Lip}}(\Xi)K^{d/2}}{\pi^{d/2}\kappa}
    \right\}.
    \end{equation}
\end{rem}

\section{\texorpdfstring{\(\chi^2\)}{Chi-square}-convergence}\label{sec:chisq}
We are now aiming at establishing the convergence of the density \(p_t\) of a solution to the stationary density \(\pi\) in \(\chi^2\)-metric defined by
\[
\chi^2(p_t\Vert\pi)
:= \int_{\R^d} 
    \frac{(p_t(x)-\pi(x))^2}{\pi(x)}\,dx.
\]
To this end, define \(f_t(x):=p_t(x)/\pi(x)\). Differentiating the \(\chi^2\)-distance with respect to time we get
\begin{eqnarray*}
    \frac{\partial}{\partial t} \chi^2(p_t\Vert\pi)
    &=&\frac{\partial}{\partial t}\int_{\R^d} 
    \frac{(p_t(x)-\pi(x))^2}{\pi(x)}\,dx
    =\frac{\partial}{\partial t}\int_{\R^d} 
    (f_t(x)-1)^2\pi(x)\,dx\\
    &=& 2 \int_{\R^d} (f_t(x)-1)\left(\partial_t f_t(x)\right)\pi(x)\,dx
    = 2 \int_{\R^d} (f_t(x)-1)\left(\partial_t p_t(x)\right)\,dx\\
    &=& 2\int_{\R^d} (f_t(x)-1)\Delta p_t(x)\,dx\\
    &&\hspace{3cm}
    +2\int_{\R^d} (f_t(x)-1) \nabla\cdot\left(
    (\nabla\Phi(x))p_t(x)\Xi(p_t(x))
    \right)\,dx\\
    &=:& I_1 + I_2,
\end{eqnarray*}
where the penultimate line uses the Fokker-Planck equation~\eqref{eq:NFPE}. For \(I_1\), formally integrating by parts, we then get 
\begin{eqnarray*}
    I_1 &=& -2\int_{\R^d} \nabla f_t(x)
    \cdot \nabla p_t(x)\,dx,
\end{eqnarray*}
provided that 
\begin{equation}\label{IBP_cond1-1}
    \nabla f_t\cdot \nabla p_t 
    =\left(
    \frac{\Vert \nabla p_t\Vert_2-p_t\nabla p_t\cdot \nabla \log \pi}{\pi}
    \right)
    \in L^1(\R^d)
\end{equation}
and 
\begin{equation}\label{IBP_cond1-2}
   \lim_{R\to\infty}
   \int_{\partial B_R}
   (f_t-1)\nabla p_t\cdot\nu\,dS=0,
   \quad 
   \nu:=x/\Vert x\Vert_2.
\end{equation}
Since \(p_t(x)=f_t(x)\pi(x)\) and therefore 
\[
\nabla p_t(x)
=\nabla f_t(x)\pi(x)
+f_t(x)\nabla \pi(x),
\]
it follows that
\begin{eqnarray*}
    I_1 &=& 
    -2\int_{\R^d} \Vert\nabla f_t(x)\Vert_2^2\pi(x)\,dx
    -2\int_{\R^d} f_t(x)\nabla f_t(x)\nabla \pi(x)\,dx\\
    &=& -2\int_{\R^d} \Vert\nabla f_t(x)\Vert_2^2\pi(x)\,dx
    + 2\int_{\R^d} f_t(x)\nabla f_t(x)\cdot(\nabla \Phi(x))\pi(x)\Xi(\pi(x))\,dx,
\end{eqnarray*}
where the last equality employs the stationary Fokker-Planck equation~\eqref{eq:stat-eq2}. Similarly, for \(I_2\) we have
\begin{eqnarray*}
    I_2 &=& 2\int_{\R^d} (f_t(x)-1) \nabla\cdot\left(
    (\nabla\Phi(x))p_t(x)\Xi(p_t(x))
    \right)\,dx\\
    &=& -2\int_{\R^d} p_t(x)\Xi(p_t(x))\nabla f_t(x)\cdot \nabla\Phi(x)\,dx\\
    &=& -2\int_{\R^d} f_t(x)\pi(x)\Xi(p_t(x))\nabla f_t(x) \cdot \nabla\Phi(x)\,dx,
\end{eqnarray*}
provided that 
\begin{equation}\label{IBP_cond2-1}
   \Xi(p_t)p_t\nabla f_t\cdot \nabla\Phi
   =\Xi(p_t)\left(
   \frac{p_t\nabla p_t-p_t^2\nabla\log\pi}{\pi}
   \right)\cdot
   \nabla\Phi
   \in L^1(\R^d)
\end{equation}
and 
\begin{equation}\label{IBP_cond2-2}
    \lim_{R\to\infty}
    \int_{\partial B_R}
    (f_t-1)(\nabla\Phi)p_t\Xi(p_t)
    \cdot\nu\,dS=0,
    \quad 
   \nu:=x/\Vert x\Vert_2,
\end{equation}
from which
\begin{multline*}
    \frac{\partial}{\partial t}\chi^2(p_t\Vert\pi)
    = -2\int_{\R^d} \Vert \nabla f_t(x)\Vert_2^2\pi(x)\,dx\\
    + 2\int_{\R^d} f_t(x) \pi(x)
    \left(
    \Xi(\pi(x))-\Xi(p_t(x))
    \right)
    \nabla f_t(x)\cdot\nabla\Phi(x)\,dx.
\end{multline*}
Similar to before, since \(\Xi\) is Lipschitz continuous with the constant \(L_{\text{Lip}}(\Xi)\), the latter integral can be upper bounded by
\begin{multline*}
    \left|
    2\int_{\R^d} f_t(x)\pi(x)
    \left(
    \Xi(\pi(x))-\Xi(p_t(x))
    \right)\nabla f_t(x)\cdot\nabla\Phi(x)\,dx
    \right|\\
    \leq 2L_{\text{Lip}}(\Xi)\int_{\R^d} f_t(x)\pi(x)|\pi(x)-p_t(x)|
    \Vert \nabla f_t(x)\Vert_2\Vert\nabla\Phi(x)\Vert_2\,dx\\
    = 2L_{\text{Lip}}(\Xi)\int_{\R^d} p_t(x)\pi(x)|f_t(x)-1|\Vert \nabla f_t(x)\Vert_2\Vert\nabla\Phi(x)\Vert_2\,dx\\
    \leq 2k_1 L_{\text{Lip}}(\Xi)\int_{\R^d} 
    \Vert\nabla f_t(x)\Vert_2^2\pi(x)\,dx\\
    + 2k_2L_{\text{Lip}}(\Xi)\int_{\R^d} \Vert\nabla \Phi(x)\Vert_2^2p_t(x)^2|f_t(x)-1|^2\pi(x)\,dx
\end{multline*}
for any constants \(k_1,k_2>0\) such that \(k_1k_2\geq 1/4\). Since by~\eqref{gaussbound_pt_unif} 
\[
\sup_{t>0,\,x\in\R^d}\Vert\nabla \Phi(x)\Vert_2^2p_t(x)^2
\leq 4c^2\sup_{x\in\R^d}\Vert x\Vert_2^2\exp\{-2\Xi(0)\Vert x\Vert_2^2\}
=\frac{2c^2}{\Xi(0) e},
\]
we arrive at
\begin{eqnarray*}
    \frac{\partial}{\partial t} \chi^2(p_t\Vert\pi)
    &\leq& 2(k_1 L_{\text{Lip}}(\Xi)-1)\int_{\R^d} \Vert \nabla f_t(x)\Vert_2^2\pi(x)\,dx\\
    &&\hspace{4cm} +\frac{4c^2k_2L_{\text{Lip}}(\Xi)}{\Xi(0) e} 
    \int_{\R^d} |f_t(x)-1|^2\pi(x)\,dx\\
    &=&2(k_1 L_{\text{Lip}}(\Xi)-1)\int_{\R^d} \Vert \nabla f_t(x)\Vert_2^2\pi(x)\,dx
    +\frac{4c^2k_2L_{\text{Lip}}(\Xi)}{\Xi(0) e}
    \chi^2(p_t\Vert\pi).
\end{eqnarray*}
Since by Lemma~\ref{LSI} the density \(\pi\) satisfies the Poincar\'e inequality with constant \(C_\mathrm{Poin}\), it follows that, provided \((k_1 L_{\text{Lip}}(\Xi)-1)<0\),
\[
\frac{\partial}{\partial t} \chi^2(p_t\Vert\pi)
\leq\left(
\frac{2(k_1 L_{\text{Lip}}(\Xi)-1)}{C_{\mathrm{Poin}}}
+\frac{4c^2k_2L_{\text{Lip}}(\Xi)}{\Xi(0) e}
\right)\chi^2(p_t\Vert\pi).
\]
Choosing \(k_1:=c\sqrt{C_{\mathrm{Poin}}/(2\Xi(0) e)}\) and \(k_2:=(4k_1)^{-1}\), we then get 
\[
\frac{\partial}{\partial t} \chi^2(p_t\Vert\pi)
\leq \left(
-\frac{2}{C_{\mathrm{Poin}}}
+L_{\mathrm{Lip}}(\Xi)
\sqrt{\frac{8c^2}{\Xi(0) e C_{\mathrm{Poin}}}}
\right)\chi^2(p_t\Vert\pi),
\]
provided that \(L_{\mathrm{Lip}}(\Xi)<\sqrt{2\Xi(0) e/(c^2C_{\mathrm{Poin}})}\). If \(L_{\mathrm{Lip}}(\Xi)\) furthermore satisfies \(L_{\mathrm{Lip}}(\Xi)<\sqrt{\Xi(0) e/(2c^2C_{\mathrm{Poin}})}\), the bracket above is strictly negative, and we conclude that 
\[
\chi^2(p_t\Vert\pi)
\leq e^{-\rho_\chi (t-t_0)}\chi^2(p_{t_0}\Vert\pi)
\]
for some \(t_0>0\) fixed and 
\[
\rho_\chi:=\frac{2}{C_{\mathrm{Poin}}}
-L_{\mathrm{Lip}}(\Xi)
\sqrt{\frac{8c^2}{\Xi(0) e C_{\mathrm{Poin}}}}.
\]

To justify~\eqref{IBP_cond1-1}--\eqref{IBP_cond2-2}, one needs a sufficient exponential decay of \(p_t\) and \(\nabla p_t\). Indeed, let \(m_1,m_2\geq 0\) and \(\alpha_1,\alpha_2>0\) be such that
\[
p_t(x)\lesssim (1+\Vert x\Vert_2)^{m_1}
e^{-\alpha_1\Vert x\Vert_2^2}
\quad \text{and} \quad
\Vert \nabla p_t(x)\Vert_2 
\lesssim (1+\Vert x\Vert_2)^{m_2}
e^{-\alpha_2\Vert x\Vert_2^2}
\]
(the constants in \(\lesssim\) possibly depending on \(t\)). From~\eqref{eq:stat-eq3}, integrating over a path \(\gamma\) from \(0\) to \(x\), we get
    \[
    \log \pi(x)=\log\pi(0)
    -\int_\gamma 2y\Xi(\pi(y))\,dy.
    \]
    Since the integral on the right-hand side is necessarily equal to \((\log\pi(0)-\log\pi(x))\), it does not depend on the choice of \(\gamma\). Fixing the contour of the form 
    \[
    y(s):=sx, \quad s\in[0,1],
    \]
    we get 
    \begin{equation}\label{pi_int}
    \log \pi(x)=\log\pi(0)
    -\int_0 ^1 2s\Vert x\Vert_2^2\Xi(\pi(sx))\,ds,
    \end{equation}
    which, given that \(\Xi(\pi(sx))\to \Xi(0)\) as \(\Vert x\Vert_2\to\infty\), implies that
    \[
    \frac{1}{\pi(x)}
    \lesssim \frac{1}{\pi(0)}
    \exp\{(\Xi(0)+\delta)\Vert x\Vert_2^2\}
    \]
    for any \(\delta>0\). Therefore, to guarantee~\eqref{IBP_cond1-1}-\eqref{IBP_cond2-2}, it is sufficient that
    \begin{eqnarray*}
        \exp\{
        (-2\alpha_2+\Xi(0)+\delta)\Vert x\Vert_2^2
        \}
        +\exp\{
        (-\alpha_1-\alpha_2+\Xi(0)+\delta)\Vert x\Vert_2^2
        \}
        &\in& L^1(\R^d),\\
        \exp\{
        (-2\alpha_1+\Xi(0)+\delta)\Vert x\Vert_2^2
        \}
        +\exp\{
        (-\alpha_1-\alpha_2+\Xi(0)+\delta)\Vert x\Vert_2^2
        &\in& L^1(\R^d),
    \end{eqnarray*}
    which yields the condition
    \begin{equation}\label{a1a2}
\min(\alpha_1,\alpha_2)>\frac{\Xi(0)}{2}.
    \end{equation}
While the bound~\eqref{gaussbound_pt_unif} provides a sufficient rate of \(\Xi(0)\) for \(p_t\), the bound~\eqref{gaussbound_grad} only yields that of \(\Xi(0)/4\) for the gradient. Hence, we now aim at establishing the corresponding estimate having a stronger rate of decay in the exponential.

The key idea is to represent the process~\eqref{eq:MV-SDE} as a sum of a linear part corresponding to the classical Ornstein-Uhlenbeck process and the non-linear input. More precisely, we decompose
\[
dX_t
=-2a_0 X_t dt 
-2(\Xi(p_t(X_t))-a_0)X_t\,dt
+\sqrt{2}\,dW_t,
\quad X_0\sim p_0,
\]
with some \(\kappa\leq a_0\leq K\), obtaining a Duhamel's formula for \(p_t\) involving an Ornstein-Uhlenbeck kernel. Given bounds~\eqref{gaussbound_pt_unif} and~\eqref{gaussbound_grad}, as well as the Gaussian bound on the initial density given by Assumption~\ref{ass:main}, we are then able to improve the rate in upper bound for \(\nabla p_t\).
More precisely, we obtain the following bound.

\begin{thm}\label{thm:grad_pt-bound}
	Assume that Assumption~\ref{ass:main} holds with \(b_0>\Xi(0)\) and \(K/\kappa\leq 2\). Then it holds that
\begin{multline}\label{grad_est}
\Vert \nabla_x p_t(x)\Vert_2
\lesssim 
\frac{c_0 \Xi(0)(1+\Vert x\Vert_2)}{\left(\Xi(0) e^{-4\Xi(0) t}+b_0(1-e^{-4\Xi(0) t})\right)^{(d+1)/2}}\\
\times \frac{1}{\sqrt{1-e^{-4\Xi(0) t}}}
\exp\left\{
-\frac{\Xi(0)b_0}{\Xi(0) e^{-4\Xi(0) t}+b_0(1-e^{-4\Xi(0) t})}\Vert x\Vert_2^2
\right\}\\
+ c_2(d,\kappa,K,b_0,c_0,\Xi(0),L_{\mathrm{Lip}}(\Xi))
(1+\Vert x\Vert_2)^3
\exp\{-\Xi(0)\Vert x\Vert_2^2\}
\end{multline}
for all \((t,x)\in(0,T)\times\R^d\), where \(c_2(d,\kappa,K,b_0,c_0,\Xi(0),L_{\mathrm{Lip}}(\Xi))>0\) is some constant.
\end{thm}

Given~\eqref{gaussbound_pt_unif} and~\eqref{grad_est}, the conditions~\eqref{IBP_cond1-1}--\eqref{IBP_cond2-2} are now justified. Combining the additional assumptions of Theorem~\ref{thm:wellposednessMV} and Lemma~\ref{lem:unifbound}, we can now state the result on \(\chi^2\)-convergence.
\begin{assumption}\label{ass2}
    Assume that Assumption~\ref{ass:main} holds with
    \begin{enumerate}
    \item \(K/\kappa\leq 2\);
    \item 
    \(
    L_{\mathrm{Lip}}(\Xi)<
    \sqrt{\frac{\Xi(0) e}{2c^2C_{\mathrm{Poin}}}};
    \)
    \item 
    \(
    b_0>\Xi(0).
    \)
    \end{enumerate}
\end{assumption}

\begin{thm}\label{thm:chisq}
Under Assumptions~\ref{ass:main} and~\ref{ass2} it holds that 
\[
\chi^2(p_t\Vert\pi)
\leq e^{-\rho_\chi (t-t_0)}\chi^2(p_{t_0}\Vert\pi)
\]
for any \(t\geq t_0>0\) and 
\[
\rho_\chi:=\frac{2}{C_{\mathrm{Poin}}}
-L_{\mathrm{Lip}}(\Xi)
\sqrt{\frac{8c^2}{\Xi(0) e C_{\mathrm{Poin}}}}.
\]
\end{thm}
\begin{rem}
Given~\eqref{C-Poin2}, for the second condition of Assumption~\ref{ass2} to hold it is sufficient that
\begin{equation}\label{Lip_aux}
L_{\mathrm{Lip}}(\Xi)
<\frac{\sqrt{e}\Xi(0)}{c}
\exp\left\{
-\frac{L_{\mathrm{Lip}}(\Xi) K^{d/2}}{2\pi^{d/2}\kappa}
\right\}.
\end{equation}
Choosing some \(\zeta\in(0,1)\) and requiring that \(L_{\mathrm{Lip}}(\Xi)K^{d/2}/(2\pi^{d/2}\kappa)<-\log\zeta\), we then get that 
\[
\frac{\sqrt{e}\Xi(0)}{c}
\exp\left\{
-\frac{L_{\mathrm{Lip}}(\Xi) K^{d/2}}{2\pi^{d/2}\kappa}
\right\}
>\zeta \frac{\sqrt{e}\Xi(0)}{c},
\]
so for~\eqref{Lip_aux} to be satisfied it is enough that \(L_{\mathrm{Lip}}(\Xi)<\zeta\sqrt{e}\Xi(0)/c\), which leads to a somewhat simpler condition
\[
L_{\mathrm{Lip}}(\Xi)
< \min\left(
\zeta \frac{\sqrt{e}\Xi(0)}{c},\,
\frac{-2\pi^{d/2}\kappa\log\zeta}{K^{d/2}}
\right).
\]
\end{rem}
\begin{rem}
Assumptions~\ref{ass:main} and~\ref{ass2} describe \(\Xi\) as a sufficiently smooth function having a rather small perturbation. Primary examples of such functions are bounded perturbations of a fixed constant. For instance, one can consider an additive perturbation
\(\Xi(u):=\kappa_1+\epsilon_1 \arctan(u)\), \(u\in\R_+\), for some \(\kappa_1>0\) and \(\epsilon_1>0\). Then \(\Xi\in C^1(\R_+)\), is Lipschitz with constant \(L_{\mathrm{Lip}}(\Xi)=\epsilon_1\), and satisfies
\[
\kappa 
\leq \kappa_1+\epsilon_1 \arctan(u)
\leq \kappa_1+\frac{\pi\epsilon_1}{2}=:K_1,
\]
where \(K_1/\kappa_1\leq 2\), provided that \(\epsilon_1<2\kappa_1/\pi\). Since \(L_{\mathrm{Lip}}(\Xi)=\epsilon_1\), requiring also that, e.g.,
\[
\epsilon_1
<\min\left(
\frac{\sqrt{e}\kappa_1}{c},\,
\frac{2\log 2 \pi^{d/2}\kappa_1}{(2\kappa_1)^{d/2}}
\right)
<\min\left(
\frac{\sqrt{e}\kappa_1}{c},\,
\frac{2\log 2 \pi^{d/2}\kappa_1}{(\kappa_1+\pi\epsilon_1/2)^{d/2}}
\right),
\]
we obtain a function satisfying our assumptions.
An example of the multiplicative perturbation is given by 
\[
\Xi(u):=\kappa_2\exp\left\{
\frac{\epsilon_2}{2}(1+\sin u)
\right\},
\quad u\in\R_+,
\]
with some \(\kappa_2>0\) and \(\epsilon_2>0\). Clearly, \(\Xi\in C^1(\R_+)\), and, since \((1+\sin u)/2\leq 1\) for any \(u\geq 0\),
\[
\kappa_2 \leq \Xi(u)
\leq \kappa_2 e^{\epsilon_2}=:K_2.
\]
Since 
\[
|\Xi'(u)|
\leq \frac{\kappa_2\epsilon_2}{2}|\cos u|
\exp\left\{
\frac{\epsilon_2}{2}(1+\sin u)
\right\}
\leq \frac{\kappa_2\epsilon_2 e^{\epsilon_2}}{2},
\]
it holds that \(L_{\mathrm{Lip}}(\Xi)\leq \kappa_2\epsilon_2 e^{\epsilon_2}/2\).
Requiring for simplicity \(\epsilon_2\leq\log 2\), one then has \(K_2\leq 2\kappa_2\) and \(L_{\mathrm{Lip}}(\Xi)\leq \kappa_2\epsilon_2\), from which 
\[
\kappa_2\epsilon_2 
< \frac{\sqrt{e}\kappa_2}{2c}
\leq \frac{\kappa_2e^{(\epsilon_2+1)/2}}{2c}
\quad\text{implies}\quad
L_{\mathrm{Lip}}(\Xi)<\frac{\sqrt{e}\Xi(0)}{2c},
\]
so that~\eqref{Lip_aux} is satisfied, e.g., if
\[
\epsilon<
\min\left(
\log 2,
\frac{e^{(\epsilon_2+1)/2}}{2c},
\frac{2\pi^{d/2}\log 2}{(2\kappa_2)^{d/2}}
\right).
\]

\end{rem}

\section{Moment bounds}
We also establish the upper bounds for the polynomial and integrated exponential moments of the solution \(X_t\) to~\eqref{eq:MV-SDE}.
\begin{lem}\label{momentcontrol}
    Under Assumption~\ref{ass:main} it holds that
    \begin{equation}\label{moments}
        \sup\limits_{t\geq 0} 
        \left(
        \mathbb{E}\left[
        \Vert X_t\Vert_2 ^p
        \right]
        \right)^{1/p}
        \leq \left(
        \sqrt{3}
        \left(
            \mathbb{E}\left[
        \Vert X_0\Vert_2^p
        \right]
        \right)^{1/p}
        +6 
        \sqrt{\frac{6C_{\mathrm{BDG}}}{\varkappa}}
        + 
        \sqrt{\frac{3d}{\varkappa}}
        \right)
        \sqrt{p}, 
        \quad p\geq 2,
    \end{equation}
    where \(C_{\mathrm{BDG}}\) is the constant from the Burkholder-Davis-Gundy inequality depending only on \(p\).
\end{lem}

\begin{lem}\label{lem:expbound_X}
    Under Assumption~\ref{ass:main}, assuming also that \(b_0>\kappa/2\), it holds that
    \[
    \E\left[ 
    \exp\left\{
    \lambda\int_0^T \Vert X_t\Vert_2^2\,dt
    \right\}
    \right]
    \leq e^{\kappa dT} 
    \E\left[\exp\left\{\frac{\kappa}{2}\Vert X_0\Vert_2^2
    \right\}\right]
    <\infty
    \]
    for any \(0\leq \lambda\leq\kappa^2\) and \(T>0\) fixed.
\end{lem}

\section{Proofs}\label{sec:proofs}

\subsection{Proof of Theorem~\ref{thm:wellposednessMV}}
For some \(T>0\), let \(q:=\{q_t\}_{0\leq t\leq T}\) be a density flow and consider the SDE
    \begin{equation}\label{MV_linear}
    dX_t^q
    =b^q (t, X_t^q)\,dt
    +\sqrt{2}dW_t,
    \quad X_0^q\sim p_0,
    \end{equation}
    with \(b^q (t, x):=-2x\Xi(q_t(x))\). On the same probability space, define the Ornstein-Uhlenbeck process 
    \begin{equation}\label{OU_a0}
    dY_t
    =-2a_0 Y_t\,dt + \sqrt{2}dW_t,
    \quad Y_0\sim p_0,
    \end{equation}
    where \(\kappa< a_0\leq K\). Defining 
    \[
    \theta_t^Y:=\sqrt{2}Y_t\left(
    a_0
    -\Xi(q_t(Y_t))
    \right).
    \]
    By Assumption~\ref{ass:main}, it holds that
    \[
    \Vert \theta_t^Y\Vert_2^2
    \leq 2(K-\kappa)^2\Vert Y_t\Vert_2^2, 
    \]
    and by Lemma~\ref{OU-expmoments}, the Novikov condition is satisfied, e.g., when \((K-\kappa)^2\leq  \kappa^2< a_0^2\), i.e., when \(K/\kappa\leq 2\). Hence, the process 
    \[
    \mathcal{E}_t^Y
    :=\exp\left\{
    \int_0^t \theta_s^Y\cdot dW_s
    -\frac{1}{2} \int_0^t \Vert \theta_s^Y\Vert_2^2\,ds
    \right\}
    \]
    is an \(\mathcal{F}_t\)-martingale under \(\mathsf{P}\), and there exists another probability measure \(\widetilde{\mathsf{P}}\) with \(d\widetilde{\mathsf{P}}/d\mathsf{P}_{\mathcal{F}_T}=\mathcal{E}_T^Y\) under which \(Y\) has the same law as \(X^q\) under \(\mathsf{P}\), i.e., \(\mathcal{L}_{\widetilde{\mathsf{P}}}(Y_{\cdot \wedge T})=\mathcal{L}_{\mathsf{P}}(X^q_{\cdot \wedge T})\). Hence, any weak solution to~\eqref{MV_linear} can be transformed to that of~\eqref{OU_a0} by the inverse Girsanov density, and by the uniqueness of the Ornstein-Uhlenbeck martingale problem, the constructed weak solution is unique. Moreover, since it holds \(\widetilde{\mathsf{P}}\ll \mathsf{P}\), it follows that \(\mathcal{L}_{\mathsf{P}}(X^q_{\cdot \wedge T})\ll \mathcal{L}_{\mathsf{P}}(Y_{\cdot \wedge T})\), and \(X^q\) possesses a (unique) density \(p^q\).

    Since by the same argument as in Lemma~\ref{momentcontrol}, we have that \(\sup_{t\geq 0}\E[\Vert X_t^q\Vert_2^p]<\infty\) for all \(p\geq 1\), provided that \(b_0>\kappa/2\), it also holds that
    \begin{multline*}
        \int_0 ^T 
        \int_{\R^d}
        \Vert b^q(t,x)\Vert_2^p
        p^q_t(x)\,dx
        =2\int_0^T 
        \int_{\R^d}
        \Vert x\Vert_2^p \Xi(q_t(x))^p
        p^q_t(x)\,dx\\
        \leq 2K^p
        \int_0^T
        \E[\Vert X_t^q\Vert_2^p]
        \,dt
        <\infty
    \end{multline*}
    for all \(p\geq 1\), in particular, \(p>d+2\). Hence, by Proposition 6.5.1 in~\cite{bogachev2022fokker} it follows that \(p^q\in C^{\beta/2,\beta}_{\mathrm{loc}}((0,T)\times\R^d)\) for some \(\beta\in(0,1)\). 
    Moreover, noticing that by Assumption~\ref{ass:main} 
    \[
    \langle 
    b^q(t,x),x
    \rangle
    =-2\Xi(q_t(x))\Vert x\Vert_2^2
    \leq -2\kappa \Vert x\Vert_2^2
    \]
    and
    \[
    \Vert b^q(t,x)\Vert_2
    \leq 2K\Vert x\Vert_2,
    \]
    for all \(x\in\R^d\), we get by Example 7.3.10 in~\cite{bogachev2022fokker} that
    \[
    p_t^q(x)\leq c_T\exp\{-\gamma_T\Vert x\Vert_2^2\},
    \quad \forall\,(t,x)\in(0,T)\times\R^d,
    \]
    for some constant \(c_T>0\) and \(\gamma_T\in(0,\kappa/2)\),
    from which it follows that
    \[
    \Vert x\Vert_2 
    p_t^q(x)
    \leq c\Vert x\Vert_2
    \exp\{-\gamma_T\Vert x\Vert_2^2\},
    \quad \forall\,(t,x)\in(0,T)\times\R^d,
    \]
    i.e.,
    \[
    \sup_{t\in(0,T),\,x\in\R^d}
    \Vert x\Vert_2 
    p_t^q(x)
    \leq \frac{c_T}{\sqrt{2e\gamma_T}}.
    \]
    Define now the class of densities
    \begin{multline*}
    \mathcal{X}_{c_T,\gamma_T}
    :=\left\{
    f \in\mathcal{Y}_{\beta}
    \colon
    f_t(x)\geq 0
    \text{ a.e.\ in } x,
    f_t(x)\leq 
    c_T\exp\{-\gamma_T\Vert x\Vert_2^2\}
    \text{ a.e.\ in } x,\right.\\ \left.
    \int_{\R^d} f_t(x)\,dx=1
    \text{ for a.e } t
    \right\},
    \end{multline*}
    where \(\mathcal{Y}_\beta:=L^{\infty}((0,T); L^2(\R^d, \Vert x\Vert_2\,dx))\), and the norm 
    \[
    \Vert f\Vert_\beta
    :=\esssup_{t\in(0,T)} 
    e^{-\beta t/2} \left(
    \int_{\R^d}
    \Vert x\Vert_2
    f_t^2(x)
    \,dx
    \right)^{1/2}
    \]
    with the induced metric
    \[
    d_\beta(f^1, f^2)
    :=\Vert f^1-f^2\Vert_\beta,
    \]
   \(\beta>0\) being is a fixed parameter the choice of which will be precised later.
    Clearly, the space \((\mathcal{Y}_\beta, \Vert\cdot\Vert_\beta)\) is a Banach space. In addition, one can observe that for any \(q\in\mathcal{X}_{c_T,\gamma_T}\), the map \(\Phi(q)=p^q\) satisfies \(\Phi(\mathcal{X}_{c_T,\gamma_T})\subset \mathcal{X}_{c_T,\gamma_T}\). Hence, existence of a weak solution to~\eqref{eq:MV-SDE} will follow from the Banach's fixed point theorem, provided that \(\mathcal{X}_{c_T,\gamma_T}\) is closed in \(\mathcal{Y}_\beta\) (which would imply that \((\mathcal{X}_{c_T,\gamma_T}, d_\beta)\) is a complete metric space), and \(\Phi\) is a contraction. 

    First, let us show the former. To this end, define an auxiliary class
    \begin{multline*}
    K_{c_T,\gamma_T}:=
    \left\{
    f\in L^2(\Vert x\Vert_2\,dx)\colon
    f(x)\geq 0
    \text{ a.e.},
    f(x) \leq c_T\exp\{-\gamma_T\Vert x\Vert_2^2\}
    \text{ a.e.},\right. \\ \left.
    \int_{\R^d} f(x)\,dx = 1
    \right\}.
    \end{multline*}
    It can be seen that \(K_{c_T,\gamma_T}\) is closed in \(L^2(\Vert x\Vert_2\,dx)\). Indeed, let \(\{f_n\}_{n\geq 1}\) be a sequence with \(f_n\in K_{c_T,\gamma_T}\) such that \(f_n\to f\) as \(n\to\infty\) in \(L^2(\Vert x\Vert_2\,dx)\). Then there exists a subsequence \(\{f_{n_k}\}_{k\geq 1}\) such that \(f_{n_k}\to f\) a.e., and, since for all \(k\geq 1\) it holds that \(f_{n_k}\in K_{c_T,\gamma_T}\), this implies that \(f_{n_k}(x)\leq c_T e^{-\gamma_T\Vert x\Vert_2^2}\) a.e., from which \(f(x)\leq c_Te^{-\gamma_T \Vert x\Vert_2^2}\) by the pointwise limit. Observing that \(c_Te^{-\gamma_T\Vert x\Vert_2^2}\in L^1(\R^d)\), we further have, by the dominated convergence, 
    \[
    \int_{\R^d} f(x)\,dx
    =\lim_{k\to\infty}
    \int_{\R^d}
    f_{n_k}(x)\,dx
    =1,
    \]
    i.e., \(f\in K_{c_T,\gamma_T}\), so \(K_{c_T,\gamma_T}\) is closed in \(L^2(\Vert x\Vert_2\,dx)\). Coming back to the original aim of showing that \(\mathcal{X}_{c_T,\gamma_T}\) is closed, take now a sequence \(\{q^n\}_{n\geq 1}\) with \(q^n\in\mathcal{X}_{c_T,\gamma_T}\) such that \(q^n\to q\) in \(\mathcal{Y}_\beta\) as \(n\to\infty\). By definition, this means
    \[
    \Vert q^n-q\Vert_\beta
    =\sup_{t\in[0,T]}
    e^{-\beta t/2}
    \Vert q^n_t-q_t \Vert_{L^2(\Vert x\Vert_2\,dx)}
    \to 0
    \quad \text{ as } n\to\infty,
    \]
    and, for any \(t\in[0,T]\) fixed,
    \[
    \Vert q_t^n-q_t\Vert_{L^2(\Vert x\Vert_2\,dx)}
    :=\left(
    \int_{\R^d}
    \Vert x\Vert_2
    \left(
    q_t^n(x)-q_t(x)
    \right)^2
    \,dx
    \right)^{1/2}
    \to 0
    \quad \text{ as } n\to\infty.
    \]
    But since for any \(t\in[0,T]\) fixed \(q_t^n\) belongs to a class \(K_{c_T,\gamma_T}\) which is closed in \(L^2(\Vert x\Vert_2\,dx)\), the latter convergence implies that the limit \(q_t\in K_{c_T,\gamma_T}\), i.e., \(q\in\mathcal{X}_{c_T,\gamma_T}\). Hence, \(\mathcal{X}_{c_T,\gamma_T}\) is closed in \(\mathcal{Y}_\beta\), and \((\mathcal{X}_{c_T,\gamma_T}, d_\beta)\) is a complete metric space. 
    
    Now let us turn towards the contractivity of \(\Phi\). Let \(q^1,q^2\in\mathcal{X}_{c_T,\gamma_T}\), \(p^{q_i}:=\Phi(q^i)\), and denote by \(P^{q_i}\) the law \(\mathcal{L}_{\mathrm{P}}(X^{q_i}_{\cdot\wedge T})\) with density \(p^{q_i}\), \(i=1,2\). Since for
    \[
    \theta_t^{q_1}
    :=\sqrt{2}
    X_t^{q_2}\left(
    \Xi(q_t^2(X_t^{q_2}))
    -\Xi(q_t^1(X_t^{q_2}))
    \right),
    \quad \forall\,0< t< T,
    \]
    it holds that
    \[
    \Vert \theta_t^{q_1}\Vert_2^2
    \leq 2(K-\kappa)^2
    \Vert X_t^{q_2}\Vert_2^2,
    \]
    by the same argument as in Lemma~\ref{lem:expbound_X}, we get that the process 
    \[
    \mathcal{E}_t^{q_1}
    :=\exp\left\{
    \int_0^t \theta_s^{q_1}\cdot dW_s
    -\frac{1}{2}\int_0^t 
    \Vert \theta_s^{q_1}\Vert_2^2\,ds
    \right\}
    \]
    is an \(\F_t\)-martingale under \(\mathrm{P}\), provided that \((K-\kappa)^2\leq \kappa^2\), and hence
    \[
    \widetilde{W}^{q_1}_t
    :=-\int_0^t \theta_s^{q_1}\,ds
    +W_t
    \]
    is an \(\F_t\)-Brownian motion under the new measure \(\mathrm{P}_1\) with \(d\mathrm{P}_1/d\mathrm{P}|_{\F_T}=\mathcal{E}_T^{q_1}\). The process \(X^{q_2}\) solves the \(q_1\)-linearised equation under this new measure. This allows to rewrite, for all \(0\leq t\leq T\),
    \begin{eqnarray*}
        KL(p^{q^1}_t\Vert p^{q^2}_t)
        &\leq&
        KL(P^{q^1}_{[0,t]}\Vert P^{q^2}_{[0,t]})
        =\E_{\mathrm{P}_1}\left[ 
        \log \mathcal{E}_t^{q_1}
        \right]\\
        &=&\E_{\mathrm{P}_1}\left[ 
        \int_0^t \theta_s^{q_2}\cdot dW_s
        -\frac{1}{2}\int_0^t 
        \Vert \theta_s^{q_2}\Vert_2^2\,ds
        \right]\\
        &=&
        \E_{\mathrm{P}_1}\left[ 
        \int_0^t \theta_s^{q_2}\cdot d\widetilde{W}^{q_2}_s
        \right]
        +\E_{\mathrm{P}_1}\left[
        \frac{1}{2}\int_0^t 
        \Vert \theta_s^{q_2}\Vert_2^2\,ds
        \right]\\
        &=&
        \E_{\mathrm{P}_1}\left[
        \int_0^t 
        \Vert X_s^{q_2}\Vert_2^2
        \left(
        \Xi(q_s^2(X_s^{q_2}))
        -\Xi(q_s^1(X_s^{q_2}))
        \right)^2
        \,ds
        \right]\\
        &\leq& 
        L^2_{\text{Lip}}(\Xi)
        \int_0^t 
        \int_{\R^d}
        \Vert x\Vert_2^2
        \left(
        q^2_s(x)-q^1_s(x)
        \right)^2
        p^{q_1}_s(x)
        \,dx\,ds.
    \end{eqnarray*} 
    Given Lemma~\ref{lem:KL_lower}, this yields
    \[
    \frac{3}{2} \int_{\R^d} 
    \frac{(p^{q_2}_s(x)-p^{q_1}_s(x))^2}{p^{q_1}_s(x)+2p^{q_2}_s(x)}\,dx
    \leq 
    L^2_{\text{Lip}}(\Xi)
        \int_0^t 
        \int_{\R^d}
        \Vert x\Vert_2^2
        \left(
        q^2_s(x)-q^1_s(x)
        \right)^2
        p^{q_1}_s(x)
        \,dx\,ds.
    \]
    Since \(p^{q_i}\in\mathcal{X}_{c_T,\gamma_T}\), we furthermore have that 
    \begin{multline*}
    \int_0^t 
        \int_{\R^d}
        \Vert x\Vert_2^2
        \left(
        q^2_s(x)-q^1_s(x)
        \right)^2
        p^{q_1}_s(x)
        \,dx\,ds\\
        \leq 
        \frac{c_T}{\sqrt{2e\gamma_T}}
        \int_0^t 
        \int_{\R^d}
        \Vert x\Vert_2
        \left(
        q^2_s(x)-q^1_s(x)
        \right)^2
        \,dx\,ds
    \end{multline*}
    and, as by the same reason \(p_s^{q_1}(x)+p_s^{q_2}(x)\leq c_T\sqrt{2/(e\gamma_T)}\Vert x\Vert_2^{-1}\), for all \(x\in\R^d\), \(s\in(0,T)\), 
    \[
    \int_{\R^d} 
    \frac{(p^{q_2}_s(x)-p^{q_1}_s(x))^2}{p^{q_1}_s(x)+2p^{q_2}_s(x)}\,dx
    \geq 
    \frac{\sqrt{e\gamma_T}}{2\sqrt{2}c_T}
    \int_{\R^d}
    \Vert x\Vert_2
    (p^{q_2}_s(x)-p^{q_1}_s(x))^2
    \,dx.
    \]
    Hence, the above leads to
    \[
    \int_{\R^d}
    \Vert x\Vert_2
    (p^{q_2}_s(x)-p^{q_1}_s(x))^2
    \,dx
    \leq 
    \frac{(2c_TL_{\text{Lip}}(\Xi))^2}{3e\gamma_T}
    \int_0^t 
        \int_{\R^d}
        \Vert x\Vert_2
        \left(
        q^2_s(x)-q^1_s(x)
        \right)^2
        \,dx\,ds.
    \]
    Since both sides of the inequality above are nonnegative, one can take square root on both sides and observe that
    \begin{multline*}
    e^{-\beta t/2}
    \left(
    \int_{\R^d}
    \Vert x\Vert_2
    (p^{q_2}_s(x)-p^{q_1}_s(x))^2
    \,dx
    \right)^{1/2}\\
    \leq 
    \frac{2c L_{\text{Lip}}(\Xi)}{\sqrt{3e\gamma_T}}
    e^{-\beta t/2}\left(
    \int_0^t 
        \int_{\R^d}
        \Vert x\Vert_2
        \left(
        q^2_s(x)-q^1_s(x)
        \right)^2
        \,dx\,ds
        \right)^{1/2}\\
        \leq \frac{2c L_{\text{Lip}}(\Xi)}{\sqrt{3e\gamma_T}}
        d_{\beta}(q^1,q^2)
        \left(
        \int_0^t 
        e^{-\beta(t-s)}
        \,ds
        \right)^{1/2}
        \leq \frac{2c L_{\text{Lip}}(\Xi)}{\sqrt{3e\gamma_T\beta}}
        d_{\beta}(q^1,q^2),
    \end{multline*}
    and, taking supremum over all \(t\in(0,T)\), we obtain
    \[
    d_{\beta}(p^{q_1}, p^{q_2})
    \leq \frac{2c_T L_{\text{Lip}}(\Xi)}{\sqrt{3e\gamma_T\beta}}
    d_{\beta}(q^1,q^2),
    \]
    i.e.,
    \[
    d_{\beta}(\Phi(q^1), \Phi(q^2))
    \leq \frac{2c_T L_{\text{Lip}}(\Xi)}{\sqrt{3e\gamma_T\beta}}
    d_{\beta}(q^1,q^2).
    \]
    Choosing \(\beta>4c_T^2L^2_{\text{Lip}}(\Xi)/(3e\gamma_T)\), we get that \(\Phi\) is a contractive map, which by the Banach's fixed point theorem implies that there exists a unique fixed point \(p=\Phi(p)\). Since by the It\^{o} formula the density of any weak solution has to satisfy~\eqref{eq:NFPE}, we get that \(p\) is the (unique) density flow of the SDE~\eqref{eq:MV-SDE}. Moreover, since \(T>0\) is arbitrary, the constructed solution is valid for any \(t> 0\).

    Finally, let us show that the constructed weak solution to~\eqref{eq:MV-SDE} is in fact a strong one. Fix some \(T>0\) and let \(X^p_t\) be a weak solution to~\eqref{MV_linear} with \(q=p\). Consider the Zvonkin-type transformation \(Y^p_t:=e^{2\Xi(0)t}X^p_t\). By the It\^{o} formula it holds that
    \[
    e^{2\Xi(0)t}X^p_t
    =X^p_0
    -2\int_0 ^t 
    e^{2\Xi(0)s}X^p_s\left(
    \Xi(p_s(X^p_s))-\Xi(0)
    \right)\,ds
    +\sqrt{2}\int_0^t e^{2\Xi(0)s}\,dW_s,
    \]
    i.e., 
    \begin{equation}\label{Zvonkin_SDE}
        dY_t^p
        =b_Y^p (t,Y_t^p)\,dt 
        +\sqrt{2} e^{2\Xi(0)t}\,dW_t,
        \quad Y_0^p\sim p_0,
        \quad t\in(0,T),
    \end{equation}
    where \(b_Y^p(t,x):=-2x(\Xi(p_t(e^{-2\Xi(0)t}x))-\Xi(0))\).
    Since the diffusion is non-degenerate and the new drift satisfies
    \[
    \Vert b^p_Y(t,x)\Vert_2
    \leq 2L_{\text{Lip}}(\Xi)
    \Vert x\Vert_2 
    p_t(e^{-2\Xi(0)t}x)
    \leq c\Vert x\Vert_2 
    \exp\{
    -\gamma e^{-4\Xi(0)t} \Vert x\Vert_2^2
    \}
    \leq \text{const},
    \]
    it follows by Theorem 1.1 in~\cite{Zhang2005} that~\eqref{Zvonkin_SDE} (and hence~\eqref{MV_linear}) admits a unique strong solution up to an explosion time. Now it remains to observe that, since
    \[
    \Vert Y_t^p\Vert_2
    \leq \Vert Y_0^p\Vert_2 
    +\text{const} \cdot T
    +\sqrt{2}\sup_{0\leq t\leq T}
    \int_0^t e^{2\Xi(0)s}\,dW_s,
    \quad \forall 0< t< T,
    \]
    there is no explosion on \((0,T)\). By the arbitrarity of \(T\), the explosion time is hence infinite. Coming back to the original equation~\eqref{eq:MV-SDE}, we conclude the claim. 

\subsection{Proof of Lemma~\ref{lem:unifbound}}
Define the function \(g\) on \(\R_+\) as
\[
g(r):=\int_1 ^r \frac{1}{s\Xi(s)},
\quad r> 0,
\quad g(0):=-\infty,
\]
and set \(\widetilde{F}(p):=p \Xi(p)\).
Then, since \(g'(r)=(r\Xi(r))^{-1}\) and on the set \(\{x\in\R^d\colon p_t(x)>0\}\) it holds that \(F(p)\nabla (g(p))=\nabla p\), on this set the equation~\eqref{eq:NFPE} can be rewritten as
\[
\partial_t p 
=\nabla \cdot \left(
\widetilde{F}(p)\nabla (g(p)+\Phi(x))
\right).
\]
Defining also for all \(t\geq 0\) the function 
\[
u_t(x):=
\begin{cases}
	g(p_t(x))+\Phi(x),&  p_t(x)>0,\\
	0,& p_t(x)=0,
\end{cases}
\]
we obtain
\[
\widetilde{F}(p)\partial_t u
=\nabla\cdot \left(
\widetilde{F}(p)\nabla u
\right)
\quad \Leftrightarrow \quad
\partial_t u
=\Delta u + \frac{\nabla \widetilde{F}(p)}{\widetilde{F}(p)}\cdot\nabla u.
\]

Now let 
\[
M_0:=\operatorname{esssup}\limits_{x\in\R^d}\left(
g(p_0(x))+\Phi(x)
\right).
\]
To establish a uniform gaussian bound for \(p_t\), we want to show that \(u_t(x)\leq M_0\). First, note that \(M_0<\infty\). Indeed, since \(\Xi\) is continuous at zero, for any \(\eps>0\) there exists a radius \(r_\eps>0\) such that \(\Xi(r)\leq \Xi(0)+\eps\) for all \(0<r<r_\eps\). Hence, for small \(r\), it holds that \(g(r)=\int_1^r (s\Xi(s))^{-1}\,ds\leq C_\eps+(\Xi(0)+\eps)^{-1}\log r\) with some constant \(C_\eps\). Therefore, for \(\Vert x\Vert_2\) large enough, given also that \(p_0(x)\leq C_0\exp\{-b_0\Vert x\Vert_2^2\}\), we get
\[
g(p_0(x))+\Phi(x)
\leq C_\eps +1+\frac{\log C_0}{\Xi(0)+\eps}
+\left(
1-\frac{b_0}{\Xi(0)+\eps}
\right)\Vert x\Vert_2^2\to -\infty,
\]
since by assumption, \(b_0>\Xi(0)\), so one can choose \(\eps>0\) such that \(b_0/(\Xi(0)+\eps>1\). Finally, since also \(p_0\in L^\infty(\R^d)\), the summand \(g(p_0)\) is bounded above, and on bounded sets the sum is finite. 

Now define 
\[
w_t(x):=(u_t(x)-M_0)_+,
\quad t\geq 0.
\]
By definition, \(w_t(x)=0\) for \(t=0\) or whenever \(p_t(x)=0\). Choose a function \(\chi_R\in C_0^\infty(\mathbb R^d)\) such that 
\begin{equation}\label{cutoff}
0\leq \chi_R(x)\leq 1,\, \forall\, x\in\R^d,
\quad \chi_R(x)\equiv 1\, \forall\, x\in B_R,
\quad \chi_R(x)\equiv 0\, \forall\, x\in \R^d\setminus B_{2R},
\end{equation}
and \(\Vert \nabla\chi_R (x)\Vert_2 \lesssim 1/R\) for all \(x\in\R^d\), where \(B_R\subset \R^d\) is a ball of radius \(R>0\). Then
\begin{multline*}
\int_{\R^d} \widetilde{F}(p_t(x))\partial_t u_t(x) w_t(x)\chi^2_R(x)\,dx
=\int_{\R^d} \operatorname{div}\left( 
\widetilde{F}(p_t(x))\nabla u_t(x)
\right)	w_t(x)\chi^2_R(x)\,dx\\
=-\int_{\R^d} \widetilde{F}(p_t(x))\nabla u\cdot\nabla\left(w_t(x)\chi^2_R(x)\right)\,dx.
\end{multline*}
Since on the set \(\{x\in\R^d\colon u_t(x)>M_0\}\) it holds that \(\nabla w_t(x)=\nabla u_t(x)\), while on \(\{x\in\R^d\colon u_t(x)\leq M_0\}\) we have \(w_t(x)=0\) and hence \(\nabla w_t(x)=0\) a.e., we arrive at
\begin{multline*}
	\int_{\R^d} \widetilde{F}(p_t(x))\partial_t u_t(x) w_t(x)\chi^2_R(x)\,dx
	=-\int_{\R^d} \widetilde{F}(p_t(x))\Vert \nabla w_t(x)\Vert_2^2\chi^2_R(x)\,dx\\
	-2\int_{\R^d} \widetilde{F}(p_t(x)) w_t(x)\chi^2_R(x)\nabla w_t(x)\cdot \nabla \chi_R\\
	=-\int_{\R^d} \widetilde{F}(p_t(x)) \Vert \chi_R\nabla w_t(x)
	+w_t(x)\nabla\chi_R(x)\Vert_2^2\,dx
	+\int_{\R^d} \widetilde{F}(p_t(x)) w_t^2(x)
	\Vert \nabla \chi_R(x)\Vert_2^2\,dx\\
	\leq \int_{\R^d} \widetilde{F}(p_t(x)) w_t^2(x)
	\Vert \nabla \chi_R(x)\Vert_2^2\,dx.
\end{multline*}
Since \(\Vert \nabla\chi_R(x)\Vert_2\lesssim R^{-1}\), and \(\nabla \chi_R\) is supported on \(R<\Vert x\Vert_2<2R\), the latter inequality implies that
\begin{equation}\label{auxmax}
\int_{\R^d} \widetilde{F}(p_t(x))\partial_t u_t(x) w_t(x)\chi^2_R(x)\,dx
\leq \frac{1}{R^2}\int_{R<\Vert x\Vert_2<2R}
\widetilde{F}(p_t(x)) w_t^2(x)\,dx.
\end{equation}
Define 
\[
E_{M_0}(z,x):=
\int_{M_0} ^z 
\widetilde{F}\left(
g^{-1}(r-\Phi(x))
\right)
(r-M_0)_+\,dr
\]
Since 
\[
\partial_z E_{M_0}(z,x)
=\widetilde{F}\left(
g^{-1}(z-\Phi(x))
\right)
(z-M_0)_+,
\]
at \(z=u_t(x)\) it holds that
\[
\partial_z E_{M_0}(u_t(x),x)
=\widetilde{F}(p_t(x))w_t(x),
\]
which implies 
\[
\partial_t E_{M_0}(u_t(x),x)
=\widetilde{F}(p_t(x))w_t(x)\partial_t u_t(x).
\]
Hence,~\eqref{auxmax} can be rewritten as
\[
\frac{d}{dt}\int_{\R^d} \widetilde{F}(p_t(x))w_t(x)u_t(x) \chi^2_R(x)\,dx
\leq \frac{1}{R^2}\int_{R<\Vert x\Vert_2<2R}
\widetilde{F}(p_t(x)) w_t^2(x)\,dx,
\]
and integrating from \(0\) to some \(t>0\) fixed, we get
\begin{multline*}
\int_{\R^d} E_{M_0}(u_t(x),x) \chi^2_R(x)\,dx\\
\leq \int_{\R^d} E_{M_0}(u_0(x),x) \chi^2_R(x)\,dx
+\frac{1}{R^2}\int_0^t\int_{R<\Vert x\Vert_2<2R}
\widetilde{F}(p_s(x)) w_s^2(x)\,dx\,ds\\
=\frac{1}{R^2}\int_0^t\int_{R<\Vert x\Vert_2<2R}
\widetilde{F}(p_s(x)) w_s^2(x)\,dx\,ds,
\end{multline*}
the last equality being due to \(u_0(x)\leq M_0\). Now, since \(g\) grows at most logarithmically in its argument, and using~\eqref{gaussbound_pt}, one has
\[
w_s(x)
=(g(p_s(x))+\Phi(x)-M_0)_+
\leq c_T (1+\Vert x\Vert_2^2),
\]
for any \(T>0\), taking \(T>t\) we get
\begin{multline*}
\frac{1}{R^2}\int_0^t\int_{R<\Vert x\Vert_2<2R}
\widetilde{F}(p_s(x)) w_s^2(x)\,dx\,ds
\leq \frac{c_T}{R^2} \int_0^t\int_{R<\Vert x\Vert_2<2R}
p_s(x)(1+\Vert x\Vert_2^2)\,dx\,ds\\
\leq c_T \int_0^t\int_{\Vert x\Vert_2>R}
p_s(x)(1+\Vert x\Vert_2^2)\,dx\,ds
\end{multline*}
Since \(\int_{\Vert x\Vert_2>R} \Vert x\Vert_2^2p_s(x)\,dx\leq R^{-2}\int_{\R^d} \Vert x\Vert_2^4p_s(x)\,dx\), and by Lemma~\ref{momentcontrol} it holds that \(\sup_{s\geq 0} \int_{\R^d} \Vert x\Vert_2^4p_s(x)\,dx<\infty\), we get 
\[
\int_{\R^d} E_{M_0}(u_t(x),x) \chi^2_R(x)\,dx \leq \eps_R
\]
with 
\[
\eps_R := c_T \int_0^t\int_{\Vert x\Vert_2>R}
p_s(x)(1+\Vert x\Vert_2^2)\,dx\,ds\to 0,
\quad R\to\infty.
\]
Since \(E_{M_0}(u_t(x),x)\chi^2_R(x)\to E_{M_0}(u_t(x),x)\) as \(R\to\infty\) for a.e.\ \(x\in\R^d\), and the functions are nonnegative, by the Fatou's lemma, 
\begin{multline*}
\int_{\R^d} E_{M_0}(u_t(x),x)\,dx
=\int_{\R^d} \liminf_{R\to\infty} E_{M_0}(u_t(x),x)\chi^2_R(x)\,dx\\
\leq \liminf_{R\to\infty} \int_{\R^d} E_{M_0}(u_t(x),x)\chi^2_R(x)\,dx
\leq \liminf_{R\to\infty} \eps_R =0.
\end{multline*}
We hence conclude that \(\int_{\R^d} E_{M_0}(u_t(x),x)\,dx=0\) and, by non-negativity, \(E_{M_0}(u_t(x),x)=0\) for a.e.\ \(x\). Since \(\widetilde{F}(g^{-1}(r-\Phi(x)))>0\) for positive arguments, it holds that \(E_{M_0}(u_t(x),x)=0\) if and only if \(u_t(x)\leq M_0\) a.e.\ \(x\) on \(\{x\in\R^d\colon p_t(x)>0\}\), and by continuity pointwise on the same set. By the definition of \(u_t\), and since by positivity of the right-hand side the inequality is trivially satisfied on \(\{x\in\R^d\colon p_t(x)=0\}\), we hence conclude 
\begin{equation}\label{pM0bound}
p_t(x)\leq g^{-1}\left(
M_0-\Phi(x)
\right)
\end{equation}
for all \(t>0\) and \(x\in\R^d\).

Now observe that, since \(\Xi\) is in \(C^1\) near zero, it follows that there exists some \(\delta>0\) such that \(\Xi(r)=\Xi(0)+O(r)\) for all \(r<\delta\). Consequently, \((s\Xi(s))^{-1}-(s\Xi(0))^{-1}=O(1)\) as \(s\to 0\), and \(\lim_{r\to 0} (g(r)-\Xi(0)^{-1}\log r)<\infty\). Thus, there exists some constant \(C>0\) such that for all \(r<\delta\) it holds that 
\[ 
g(\delta)
\geq 
g(r)\geq \frac{1}{\Xi(0)}\log r-C,
\]
where the first inequality follows from the fact that \(g\) is non-decreasing.
Taking \(g(r):=z\) we get, for all \(z<z_0\) with \(z_0:=g(\delta)\),
\[
z\geq \frac{1}{\Xi(0)}\log g^{-1}(z)-C
\quad \Leftrightarrow\quad 
g^{-1}(z)
\leq e^C e^{\Xi(0)z}.
\]
Applying this to~\eqref{pM0bound}, we then get
\[
p_t(x)\leq e^C \exp\{
\Xi(0)
(M_0-\Phi(x))
\}
\lesssim C_1(b_0, c_0, \Xi(0))\exp\{-\Vert x\Vert_2^2\},
\]
whenever \(M_0-\Phi(x)<z_0\), i.e., for all \(\Vert x\Vert_2>M_0-1-z_0\). For all \(\Vert x\Vert_2\leq M_0-1-z_0\), one observes that
\begin{multline*}
p_t(x)
\leq g^{-1}(M_0-\Phi(x))e^{\Xi(0)(M_0-1-z_0)}\exp\{-\Xi(0)\Vert x\Vert_2^2\}\\
\lesssim C_2(b_0, c_0, \Xi(0))\exp\{-\Xi(0)\Vert x\Vert_2^2\},
\end{multline*}
which together with the previous bound yields the result for all \(x\in\R^d\). 

\subsection{Proof of Theorem~\ref{C2+grad}}
We start with establishing the local regularity of the weak solution constructed in Theorem~\ref{thm:wellposednessMV}. Since by Theorem~\ref{thm:wellposednessMV} it is known that \(p\) belongs to \(C^{\beta/2,\beta}_{\mathrm{loc}}((0,T)\times\R^d)\) for some \(\beta\in(0,1)\), the core idea is to employ the parabolic bootstrap.
Denote 
\[
F(x,t):=2x\,p_t(x)\Xi(p_t(x)).
\]
Since \(p\in C^{\beta/2,\beta}_{\mathrm{loc}}
\bigl((0,T)\times\mathbb R^d\bigr)\), 
\(\Xi\in C^1(\R_+)\), and \(x\) is bounded on compact subsets of \(\mathbb R^d\), we have
\[
F\in C^{\beta/2,\beta}_{\mathrm{loc}}
\bigl((0,T)\times\mathbb R^d\bigr).
\]
Hence, rewriting~\eqref{eq:NFPE} as
\[
\partial_t p_t-\Delta p=\operatorname{div}F,
\]
and identifying \(F\) as \(G\) in Lemma~\ref{lem:heat_aux} (the second summand is then \(H=0\)), we get that \(\nabla_x p\in C^{\beta/2,\beta}_{\mathrm{loc}}(Q)\) for any bounded cylinder \(Q\in (0,T)\in\mathbb{R}^d\), from which \(\nabla_x p\in C^{\beta/2,\beta}_{\mathrm{loc}}((0,T)\in\mathbb{R}^d)\). Expanding now the divergence 
\[
\operatorname{div}(2xp\Xi(p))
=
2d\,p\Xi(p)
+
2x\cdot(\Xi(p)+p\Xi'(p))\nabla_x p
=:f,
\]
we notice that \(\partial_t p_t-\Delta p=f\) with \(f\in C^{\beta/2,\beta}_{\mathrm{loc}}((0,T)\times \R^d)\). Choose again some nested bounded cylinders \(Q_r\Subset Q_{2r}\Subset (0,T)\times\R^d\) and a function \(\chi\in C_0^\infty(Q_{2r})\) such that \(\chi\equiv 1\) on \(Q_r\). Since \(p,f\in C^{\beta/2,\beta}(Q_{2r})\), the function \(R:=\chi(f+p)\) belongs to \(C^{\beta/2,\beta}_0(Q_{2r})\). By Theorem 8.7.3 in~\cite{krylov1996lectures}, there exists a unique function \(u\in C^{1+\beta/2,2+\beta}(\R^{d+1})\) satisfying 
\[
\Delta u-\partial_t u_t-u=-R
\]
in \(\R^{d+1}\). Rewriting now~\eqref{eq:NFPE} as
\[
\Delta p-p_t-p=-(f+p)
\]
and noticing that on \(Q_r\) this is equivalent to
\[
\Delta p-p_t-p=-R,
\]
we again observe that \(h:=p-u\) satisfies \(\Delta h-\partial_t h_t-h=0\) on \(Q_r\), and, by hypoellipticity of \(\Delta-\partial_t-1\), it follows that \(h\in C^\infty_{\mathrm{loc}}(Q_r)\). Since \(p=u+h\) and \(u\in C^{1+\beta/2,2+\beta}(Q_r)\), we conclude that \(p\in C^{1+\beta/2,2+\beta}_{\mathrm{loc}}(Q_r)\) and, since the cylinder was arbitrary, \(p\) belongs to \(C^{1+\beta/2,2+\beta}_{\mathrm{loc}}((0,T)\times\R^d)\).

Now let us turn towards the gradient bound. Fix some \(t_0\in(0,T)\) and \(x_0\in\R^d\) and define 
\[
\theta:=\theta(t_0)=\min(1,\sqrt{t_0/2})
\quad \text{and} \quad 
r:=r(x_0)=\frac{\theta}{1+\Vert x_0\Vert_2}.
\]
Since under this choice \(r^2<t_0/2\) and \(t_0-r^2\geq t_0/2>0\), the cylinder \(Q^-_r(t_0,\,x_0):=(t_0-r^2,\,t_0]\times B_r(x_0)\) with \(B_r(x_0)\) being a ball of radius \(r\) in \(\R^d\) centered in \(x_0\) is contained in \((0,T)\times \R^d\). 
Defining
\begin{equation*}
 s:=\frac{t-t_0}{r^2}+1+\epsilon_0,
 \qquad y:=\frac{x-x_0}{r},
\end{equation*}
with some \(\epsilon_0\in(0,0.5)\) one can observe that \((t,x)\in Q_r^-(t_0,\,x_0)\) if and only if \((s,y)\in (\epsilon_0,1+\epsilon_0]\times B_1(0)\). With this idea in mind, define also the rescaled version of \(p\) by
\begin{equation*}
 v(s,y):=\frac{p_{t_0+r^2(s-1-\epsilon_0)}(x_0+ry)}{S},
 \quad (s,y)\in (\epsilon_0,1+\epsilon_0]\times B_1(0),
\end{equation*}
where 
\[
 S:=\sup_{(s,y)\in (\epsilon_0,1+\epsilon_0]\times B_1(0)} p_{t_0+r^2(s-1-\epsilon_0)}(x_0+ry).
\]
Rewriting~\eqref{eq:NFPE} as
\[
\partial_t p_t-\operatorname{div}(\nabla p+2xp\Xi(p))=0, \quad p(0,x)=p_0(x),
\]
with \((t,\,x):=(t_0+r^2 (s-1-\epsilon_0),\, x_0+ry)\in Q_r^-(t_0,\,x_0)\),
and noticing that 
\[
\partial_s v_s=\frac{r^2}{S}\partial_t p_{t_0+r^2(s-1-\epsilon_0)},
 \quad
 \nabla v_s(y)=\frac{r}{S}(\nabla p_{t_0+r^2(s-1-\epsilon_0)})(x_0+ry),
 \quad
 \operatorname{div}_x=\frac{1}{r}\operatorname{div}_y,
\]
we conclude that \(v\) satisfies
\[
\partial_s v
 -\operatorname{div}\bigl(\nabla v+2r(x_0+ry)v\Xi(Sv)\bigr)=0,\,\,
  v(0,y)=v_0(y),\,\,
\forall (s,y)\in (\epsilon_0,1+\epsilon_0]\times B_1(0).
\]
Since this is now a quasi-linear equation with bounded coefficients, we get by Theorem 3.1 of Chapter 5 in~\cite{LSU} that 
on a smaller cylinder \((2\epsilon_0,1+\epsilon_0/2]\times B_{1/2}(0)\) it holds that
\begin{equation*}
 \max_{(2\epsilon_0,1+\epsilon_0/2]\times B_{1/2}(0)}\Vert \nabla v\Vert_2
 \leq C(d,c,K,L_{\text{Lip}}(\Xi))
\end{equation*}
with some constant growing at most polynomially in each of the variables. In particular,
\begin{equation*}
 \Vert \nabla v(1,0)\Vert_2\leq C(d,c,K,L_{\text{Lip}}(\Xi)).
\end{equation*}
Recalling that
\begin{equation*}
 v(1,0)=\frac{p_{t_0}(x_0)}{S},
 \qquad
 \nabla v(1,0)=\frac{r}{S}(\nabla p_{t_0})(x_0),
\end{equation*}
we hence get
\begin{equation*}
 \Vert \nabla p_{t_0}(x_0)\Vert_2\leq \frac{C(d,c,K,L_{\text{Lip}}(\Xi))}{r}S,
\end{equation*}
i.e.,
\begin{eqnarray*}
 \Vert\nabla p_{t_0}(x_0)\Vert_2
 &\leq& \frac{C(d,c,K,L_{\text{Lip}}(\Xi))}{\theta}(1+\Vert x_0\Vert_2)\sup_{(s,y)\in (\epsilon_0,1+\epsilon_0]\times B_{1}(0)}p_{t_0+r^2(s-1)}(x_0+ry)\\
 &\leq& \frac{C(d,c,K,L_{\text{Lip}}(\Xi))}{\theta}(1+\Vert x_0\Vert_2)\sup_{y\in B_1(0)}\exp\{-\Xi(0)\Vert x_0+ry\Vert_2^2\}.
\end{eqnarray*}
Since
\begin{equation*}
 \Vert x_0+ry\Vert_2^2
 \geq (\Vert x_0\Vert_2-r)^2
 =\left(\Vert x_0\Vert_2-\frac{\theta}{1+\Vert x_0\Vert_2}\right)^2,
\end{equation*}
by $\theta\leq 1$, we get, whenever $\Vert x_0\Vert_2\geq 1$, that $r\leq \Vert x_0\Vert_2/2$, so that
\begin{equation*}
 \Vert x_0+ry\Vert_2^2
 \geq \left(\Vert x_0\Vert_2-\frac{\theta}{1+\Vert x_0\Vert_2}\right)^2
 \geq \frac{\Vert x_0\Vert_2^2}{4},
\end{equation*}
from which
\begin{equation*}
 \sup_{y\in B_1(0)}c\exp\{-\Xi(0)\Vert x_0+ry\Vert_2^2\}
 \leq c\exp\left\{-\frac{\Xi(0)}{4}\Vert x_0\Vert_2^2\right\}.
\end{equation*}
Whenever $\Vert x_0\Vert_2<1$, note that
\begin{equation*}
 \exp\left\{-\frac{\Xi(0)}{4}\Vert x_0\Vert_2^2\right\}
 \ge \exp\{-\Xi(0)/4\},
\end{equation*}
which allows the estimate
\begin{equation*}
 \sup_{y\in B_1(0)} c\exp\{-\Xi(0)\Vert x_0+ry\Vert_2^2\}
 \leq c
 \leq ce^{\Xi(0)/4}\exp\left\{-\frac{\Xi(0)}{4}\Vert x_0\Vert_2^2\right\}.
\end{equation*}
Hence, we arrive at
\begin{multline*}
 \Vert \nabla p_{t_0}(x_0)\Vert_2
 \leq \frac{c_1(d,K,b_0, c_0, \Xi(0),L_{\text{Lip}}(\Xi))}{\theta}(1+\Vert x_0\Vert_2)
 \exp\left\{-\frac{\Xi(0)}{4}\Vert x_0\Vert_2^2\right\}\\
 = \frac{c_1(d,K,b_0, c_0, \Xi(0),L_{\text{Lip}}(\Xi))}{\min(1,\sqrt{t/2})}(1+\Vert x_0\Vert_2)
 \exp\left\{-\frac{\Xi(0)}{4}\Vert x_0\Vert_2^2\right\}.
\end{multline*}
Since \(t_0\in(0,T)\) and \(x_0\in\R^d\) were arbitrary, we conclude the claim.

\subsection{Proof of Lemma~\ref{lem:stationary}}
By definition, any stationary density \(\pi\) must satisfy the stationary
version of (\ref{eq:NFPE}), i.e.,
\begin{eqnarray}
\label{eq:stat-eq2}
\Delta \pi+\mathrm{div}(\nabla\Phi\, \Xi(\pi)\pi)=0.
\end{eqnarray}
Let \(\pi\) be any stationary probability density solving~\eqref{eq:stat-eq2} weakly. Equivalently, \(\pi\) is a probability solution to
\begin{equation}\label{eq:stat-eq3}
\mathrm{div}(\nabla\pi+\nabla\Phi \Xi(\pi)\pi)=0,
\end{equation}
i.e., \(\pi\) satisfies
\begin{equation}\label{L_pi}
\int_{\R^d} \pi(x) L_\pi \varphi(x)\,dx=0,
\quad \forall\,\varphi\in C^\infty_0(\R^d),
\end{equation}
with \(L_\pi:=\Delta-2x\Xi(\pi)\cdot\nabla\). Fix some \(a\in(0,\kappa)\) and define \(V(x):=e^{a\Vert x\Vert_2^2}\). Since \(\Xi(\pi(x))\geq\kappa\) for all \(x\in\R^d\), it holds that
\begin{multline}\label{L_pi_bound}
L_\pi V(x)
=(2ad+4a^2\Vert x\Vert_2^2)V(x)
-4a\Vert x\Vert_2^2\Xi(\pi(x))V(x)\\
\leq \left(
2ad-4a(\kappa-a)\Vert x\Vert_2^2
\right)
V(x)
\end{multline}
for any \(x\in\R^d\), where \((\kappa-a)>0\). Hence, the coefficient before \(V\) is negative for \(\Vert x\Vert_2\) large enough, and there exist constants \(k, k_R, R>0\) such that
\[
L_\pi V(x)\leq k_R \mathbbm{1}_{B_R(0)}(x)-k\Vert x\Vert_2^2 V(x)\mathbbm{1}_{\R^d\setminus B_R(0)}(x),
\]
where \(B_R(0)\subset\R^d\) is a ball of radius \(R\) centered in zero. Now let \(\{f_n\}_{n\geq 1}\) be a sequence of increasing and concave functions \(f_n\in C^\infty([0,\infty))\) with
\[
0\leq f_n'(r)\leq 1,
\quad f_n''(r)\leq 0,
\quad \forall\, r\geq 0,
\]  
and such that \(f_n(r)\uparrow 1\) as \(r\to\infty\), \(f_n(r)=r\) for \(r\leq n\) and \(f_n=2n\) for \(r>2n\) (for example, such functions can be constructed from a nonincreasing cutoff). Since \((f_n(V(x))-f_n(2n))\in C^\infty_0(\R^d)\), one can use it as a function \(\varphi\) in the weak formulation~\eqref{L_pi}, leading to
\begin{multline*}
0=\int_{\R^d} \pi(x) L_\pi\left(f_n(V(x))-f_n(2n)\right)\,dx
=\int_{\R^d} \pi(x) L_\pi f_n(V(x))\,dx\\
=\int_{\R^d} \pi(x)\left(
f'_n(V(x))\Delta V(x)+f''_n(V(x))\Vert \nabla V(x)\Vert_2^2 \right.\\\left.
\hspace{8cm}-2x\Xi(\pi(x))\cdot f'_n(V(x))\nabla V(x)
\right)\,dx\\
=\int_{\R^d} \pi(x)\left( 
f'_n(V(x))L_\pi V(x) + f''_n(V(x))\Vert \nabla V(x)\Vert_2^2
\right)\,dx\\
\leq \int_{\R^d} \pi(x)f'_n(V(x))L_\pi V(x)\,dx
\end{multline*}
by the non-positivity of \(f''_n\). Using~\eqref{L_pi_bound} we then get
\[
0\leq k_R \int_{B_R(0)} \pi(x)\,dx 
- k\int_{\R^d\setminus B_R(0)} \Vert x\Vert_2^2 V(x) \pi(x) f'_n(V(x))\,dx,
\]
and, since \(\pi\) being a probability density satisfies \(\int_{B_R(0)}\pi(x)\,dx\leq 1\), it follows that 
\[
\int_{\R^d\setminus B_R(0)} \Vert x\Vert_2^2 V(x) \pi(x) f'_n(V(x))\,dx \leq \frac{k_R}{k}.
\]
Letting \(n\to\infty\), by monotone convergence, we get that \(\int_{\R^d\setminus B_R(0)}\Vert x\Vert_2^2 e^{a\Vert x\Vert_2^2}\pi(x)\,dx<\infty\), implying that, for every \(a\in(0,\kappa)\)
 \[
\int_{\mathbb R^d}e^{a|x|^2}\pi(x)\,dx<\infty,
\quad\text{in particular,}\quad
\int_{\mathbb R^d}(1+|x|)^m e^{a|x|^2}\pi(x)\,dx<\infty,
\,\,\forall\,m\geq 0.
\]
Now let us note that the measure \(\Pi_t(dx):=\pi(x)\,dx\) corresponding to \(\pi\) can be viewed as a stationary solution to the equation \(\partial_t\Pi_t=L^*_\pi \Pi_t\). Define \(\Lambda(x):=e^{\alpha\Vert x\Vert_2^2}\) with some \(0<\alpha<a<\kappa\) and observe that \(\widetilde{\Pi}_t(dx):=\Lambda(x)\Pi_t(x)/\int_{\R^d}\Lambda(x)\Pi_t(x)\) is itself a probability measure solving \(\partial_t \widetilde{\Pi}_t=\widetilde{L}_\pi^*\widetilde{\Pi}_t\) with
\begin{multline*}
\widetilde{L}_\pi \varphi
=
\left(\Delta\varphi + (-2x\Xi(\pi(x))+4\alpha x)\cdot\nabla\varphi
+(2\alpha d+4\alpha^2\Vert x\Vert_2^2-4\alpha\Vert x\Vert_2^2\Xi(\pi(x))\varphi\right)\\
\times \left(\int_{\R^d}\Lambda(x)\mu_t(x)\right)
=0,
\end{multline*}
where the coefficients are locally \(L^p\) in space.
By Theorem 7.3.3 in~\cite{bogachev2022fokker}, it follows that that \(\Lambda(x)\pi(x)\leq \mathrm{const}\) for all \(x\in\R^d\), so \(\pi(x)\lesssim e^{-\alpha\Vert x\Vert_2^2}\) for all \(x\in\R^d\) pointwise. Hence, \(\pi\) satisfies condition 3 of Assumption~\ref{ass:main}, which makes it an admissible initial density \(p_0\) in Theorem~\ref{thm:wellposednessMV}. Applying this theorem to~\eqref{eq:NFPE} starting from \(\pi\), we then conclude that the resulting solution is unique, and by Theorem~\ref{C2+grad} it holds that \(\pi\in C^2_{\mathrm{loc}}(\R^d)\). Moreover, by the strong maximum principle (see, e.g., Theorem 2.7 in~\cite{lieberman1996}, which can be applied to \(u(t,x):=-e^{-\lambda t}\pi(x)\) with \(\lambda\geq 2dK\)), it follows that \(\pi(x)>0\) everywhere on \(\R^d\).

While the previous argument restricts the class of possible stationary density to \(C^2_{\mathrm{loc}}\) with Gaussian tails, it does not forbid two different stationary solutions belonging to this class. We hence aim at proving that the stationary density is unique. To this end, define \(g\) as in~\eqref{g}. Note that the integrand is well-defined by the strong positivity of \(\pi\). By the assumption on \(\Xi\), it can be seen that \(g\in C^1((0,\infty))\) and \(g'(r)=(r\Xi(r))^{-1}\) for all \(r>0\), implying that \(g\) is monotonically increasing, and hence \(g^{-1}\) is also well-defined on \((0,\infty)\). One can now observe that 
\[
\nabla g(\pi(x))
=\frac{1}{\pi(x)\Xi(\pi(x))}\nabla \pi(x),
\]
from which, denoting \(F(x):=g(\pi(x))+\Phi(x)\), it follows that
\[
\nabla \pi+2x\pi\Xi(\pi)
=\pi\Xi(\pi)\nabla F,
\]
i.e.,~\eqref{eq:stat-eq2} can be rewritten as
\begin{equation}\label{eq:stat-eq4}
\mathrm{div}(\pi\Xi(\pi)\nabla F)=0.
\end{equation}
Observe that, by the bounds on \(\Xi\),
\[
|g(u)|\leq \mathrm{const}(1+|\log u|), 
\quad u>0,
\]
while by the pointwise Gaussian bound on \(\pi\)
\[
\int_{\R^d} \Vert x\Vert_2^4 \pi(x)\,dx<\infty
\quad \text{and} \quad
\int_{\R^d} \pi(x)|\log\pi(x)|^2\,dx<\infty,
\]
implying that
\(
\int_{\R^d} \pi(x)F(x)^2\,dx<\infty. 
\)
Choose a function \(\chi_R\in C_0^\infty(\mathbb R^d)\) as in~\eqref{cutoff}
and for some mollifier \(\rho_\epsilon\) with \(\epsilon>0\) define \(\varphi_{R,\epsilon}:=\rho_\epsilon\star \varphi_R\) with \(\varphi_R:=\chi_R^2 F\). Since \(\varphi_{R,\epsilon}\in C^\infty_0(\R^d)\), it is an admissible function in the weak formulation of~\eqref{eq:stat-eq4}, and since \(\varphi_{R,\epsilon}\to \varphi_R\) in \(C_0^1(\R^d)\) as \(\epsilon\to 0\), one can pass to the limit as \(\epsilon\to 0\), arriving at
\begin{multline*}
    0
    =-\int_{\R^d} 
    \pi(x)\Xi(\pi(x))\nabla F(x)\cdot \nabla (\chi_R^2(x)F(x))\,dx\\
    \Leftrightarrow 
    \int_{\R^d}
    \chi_R(x)^2 \pi(x)\Xi(\pi(x))\Vert \nabla F(x)\Vert_2^2\,dx\\
    =-2\int_{\R^d} \chi_R(x)\pi(x)\Xi(\pi(x))F(x)\nabla F(x)\cdot \nabla \chi_R(x)\,dx
\end{multline*}
Since by the Young's inequality
\begin{multline*}
    \int_{\R^d} |2\chi_R(x)\pi(x)\Xi(\pi(x))F(x)\nabla F(x)\cdot \nabla \chi_R(x)|\,dx\\
    \leq \frac{1}{2}\int_{\R^d} \chi_R^2(x)\pi(x)\Xi(\pi(x))\Vert\nabla F(x)\Vert_2^2\,dx
    +2\int_{\R^d} \pi(x)\Xi(\pi(x))F(x)^2\Vert \nabla\chi_R(x)\Vert_2^2\,dx,
\end{multline*}
it follows from the above that 
\begin{eqnarray*}
\int_{\R^d} \chi_R^2(x)\pi(x)\Xi(\pi(x))\Vert\nabla F(x)\Vert_2^2\,dx
&\leq& 4 \int_{\R^d} \pi(x)\Xi(\pi(x))F(x)^2\Vert \nabla\chi_R(x)\Vert_2^2\,dx\\
&\leq& \frac{4K}{R^2}\int_{B_{2R}\setminus B_R} \pi(x)F(x)^2\,dx,
\end{eqnarray*}
with the last inequality following from the upper bounds for \(\Xi\) and \(\Vert \nabla\chi_R\Vert_2\). Since \(\pi F^2\in L^1(\R^d)\), the right-hand side tends to zero as \(R\to\infty\). Fixing any ball \(B_m\) with \(R>m\), we then get, since \(\chi_R\equiv 1\) on \(B_m\), that
\[
\int_{B_m} \pi(x)\Xi(\pi(x))\Vert \nabla F(x)\Vert_2^2\,dx
\leq \frac{4K}{R^2}\int_{B_{2R}\setminus B_R} \pi(x)F(x)^2\,dx,
\]
where the right-hand side tends to zero as \(R\to\infty\), yielding \(\int_{B_m} \pi(x)\Xi(\pi(x))\Vert \nabla F(x)\Vert_2^2\,dx=0\). Since \(m>0\) was arbitrary, it follows that \(\pi(x)\Xi(\pi(x))\Vert \nabla F(x)\Vert_2^2=0\) on \(\R^d\), and since \(\pi(x)\Xi(\pi(x))>0\), we conclude that \(\nabla F(x)=0\) on \(\R^d\). Therefore any stationary density satisfies 
\[
F(x)=\mu
\quad \Leftrightarrow\quad
\pi(x)=g^{-1}(\mu-\Phi(x)),
\]
where \(\mu\in\R\) is a normalisation constant ensuring that \(\pi\) integrates to unity. Again by the bounds on \(\Xi\), it holds that
\[
\begin{cases}
    e^{\kappa u} \leq g^{-1}(u) \leq e^{Ku},& u\geq 0,\\
    e^{K u} \leq g^{-1}(u) \leq e^{\kappa u},& u\leq 0,
\end{cases}
\]
from which it follows that \(\pi(x)=g^{-1}(\mu-\Phi(x))\leq c_1 e^{-c_2\Vert x\Vert_2^2}\) for some constants \(c_1,c_2>0\). Defining now \(M(\mu):=\int_{\R^d} g^{-1}(\mu-\Phi(x))\,dx\in(0,\infty)\) and observing that \(M(\mu)\to 0\) as \(\mu\to -\infty\) and \(M(\mu)\to\infty\) as \(\mu\to\infty\), we conclude that there exists a unique constant \(\mu\in\R\) such that \(M(\mu)=1\), i.e., such that \(\pi\) is indeed a probability density. 

\subsection{Proof of Lemma~\ref{LSI}}
    Denote \(s:=\Vert x\Vert_2^2\) and define \(G(s):=-\log g^{-1}(\mu-1-s)\) with \(g\), \(\mu\) given in Lemma~\ref{lem:stationary}. Since \(g(g^{-1}(\mu-1-s))=\mu-1-s\), differentiating both sides with respect to \(s\) one obtains that
    \[
    (g^{-1}(\mu-1-s))'
    =-g^{-1}(\mu-1-s)\Xi(g^{-1}(\mu-1-s)),
    \]
    from which \(G'(s)=\Xi(g^{-1}(\mu-1-s))\). Since by Assumption~\ref{ass:main} the function \(\Xi\) is bounded between \(\kappa\) and \(K\), it follows that
    \[
    \kappa\leq G'(s)\leq K
    \quad \Leftrightarrow \quad
    G(0)+\kappa s
    \leq G(s)
    \leq G(0)+Ks,
    \quad s\geq 0.
    \]
    Changing back from \(s\) to \(x\), we conclude~\eqref{pi_bound}. 

    Now define \(\widetilde{G}(s):=\Xi(0)s\). Since \(g^{-1}(1-\mu-s)\leq \pi(0)e^{-\kappa s}\), it holds that
    \begin{eqnarray*}
        \operatorname{osc}\limits_{s\geq 0}
        (G(s)-\widetilde{G}(s))
        &=&\sup\limits_{s\geq 0}
        (G(s)-\widetilde{G}(s))
        -\inf\limits_{s\geq 0}
        (G(s)-\widetilde{G}(s))\\
        &=&\sup_{s,t\geq 0}\left|
        (G(s)-\widetilde{G}(s))
        -(G(t)-\widetilde{G}(t))
        \right|\\
        &\leq& \int_0^\infty |G'(r)-\widetilde{G}'(r)|\,dr
        = \int_0^\infty |\Xi(g^{-1}(\mu-1-r))-\Xi(0)|\,dr\\
        &\leq& L_{\mathrm{Lip}}(\Xi)\int_0^\infty |g^{-1}(\mu-1-r)|\,dr
        \leq L_{\mathrm{Lip}}\pi(0)\int_0^\infty e^{-\kappa r}\,dr\\
        &=&\frac{L_{\mathrm{Lip}}\pi(0)}{\kappa}.
    \end{eqnarray*}
    Since \(\widetilde{G}(s)\) corresponds to the Gaussian potential with exponent \(\widetilde{V}(x)=-\Xi(0)\Vert x\Vert_2^2\), for which \(\nabla^2 \widetilde{V}(x)=2\Xi(0)I_d\) and hence the Log-Sobolev inequality is satisfied with constant \(C_{\mathrm{LSI}}^{\widetilde{V}}:=(2\Xi(0))^{-1}\), it follows by the Holley-Strook perturbation lemma (see, e.g., Theorem 1.1 in~\cite{cattiaux2022functional}) that \(\pi\) satisfies this inequality with constant \(C_\mathrm{LSI}\leq e^{\operatorname{osc}(G-\widetilde{G})}C_{\mathrm{LSI}}^{\widetilde{V}}\), i.e., \(C_\mathrm{LSI}=
    (2\Xi(0))^{-1}
   \exp\left\{
   L_{\mathrm{Lip}}(\Xi)\pi(0)/\kappa
   \right\}\). By the relation between the Log-Sobolev and Poincar\'e inequalities, we get the claim.

\subsection{Proof of Theorem~\ref{thm:grad_pt-bound}}
The key idea is to employ the Duhamel's formula~\eqref{p_duh} with \(a_0=\Xi(0)\) and differentiate it in \(x\). 
Decompose~\eqref{eq:NFPE} as in~\eqref{eq:split}. By Assumption~\eqref{ass:main} alongside~\eqref{gaussbound_pt_unif} and~\eqref{gaussbound_grad} it holds that
\begin{eqnarray}\label{nablaH}
|\nabla_y \cdot H(s,y)|
&=&| 2d(\Xi(p_s(y))-\Xi(0))p_s(y) +2y\cdot (\nabla_y(\Xi(p_s(y))-\Xi(0)))p_s(y)\nonumber \\
&&
+2y(\Xi(p_s(y))-\Xi(0))\cdot\nabla_y p_s(y)| \nonumber \\
&\leq& 2dL_{\mathrm{Lip}}(\Xi)p_s(y)\left(
p_s(y)
+2\Vert y\Vert_2 \Vert \nabla_y p_s(y)\Vert_2
\right) \nonumber\\
&\leq&
2dc L_{\mathrm{Lip}}(\Xi)\left(
c+\frac{2c_1 \Vert y\Vert_2(1+\Vert y\Vert_2)}{\min(1,\sqrt{s/2})}
\right)\exp\left\{
-(\Xi(0)+\Xi(0)/4)\Vert y\Vert_2^2 
\right\} \nonumber \\
&\lesssim& 
\frac{c_2(1+\Vert y\Vert_2)^{2}}{\min(1,\sqrt{s/2})}
\exp\left\{
-(5\Xi(0)/4)\Vert y\Vert_2^2 
\right\},
\end{eqnarray}
where \(c_1:=c_1(d,K,b_0,c_0,\Xi(0),L_{\mathrm{Lip}}(\Xi))\) is a constant from~\eqref{gaussbound_grad}, and \(c_2\) is the enlarged constant, \(c_2:=c_2(d,\kappa,K,b_0,c_0,\Xi(0),L_{\mathrm{Lip}}(\Xi))\).  
We now justify differentiating the Duhamel term in \(x\). To this end, we are bounding the time and space variables. More precisely, let \(R>0\) and choose a function \(\chi_R\in C_0^\infty(\mathbb R^d)\) as in~\eqref{cutoff}. For \(\varepsilon>0\), define
\[
I_{\varepsilon,R}(t,x)
:=
\int_0^{t-\varepsilon}
\int_{\mathbb R^d}
G_{t-s}(x,y)\chi_R(y)\nabla\cdot H(s,y)\,dy\,ds.
\]
For fixed \(R\) and \(\varepsilon\), the integration is over
\((0,t-\varepsilon)\times B_{2R}\), where the kernel is smooth in \(x\) and \(\nabla_x G_{t-s}(x,y)\) is bounded locally in \(x\). Since \(\nabla\cdot H\in L^1_{\mathrm{loc}}((0,T)\times\R^d)\), one can differentiate under the integral sign, which yields
\[
\nabla_x I_{\varepsilon,R}(t,x)
=
\int_0^{t-\varepsilon}
\int_{\mathbb R^d}
\nabla_xG_{t-s}(x,y)\chi_R(y)\nabla\cdot H(s,y)\,dy\,ds.
\tag{3}
\]
We first let \(R\to\infty\). By~\eqref{nablaH},
\begin{multline}\label{nablaGnablaH}
\Vert \nabla_x G_{t-s}(x,y)\chi_R(y)\nabla_y\cdot H_s(y)\Vert_2
\lesssim 
\Vert \nabla G_{t-s}(x,y)\Vert_2
\frac{c_2(1+\Vert y\Vert_2)^{2}}{\min(1,\sqrt{s/2})}\\
\times
\exp\left\{
-5\Xi(0)/4)\Vert y\Vert_2^2 
\right\}.
\end{multline}
Since \(s\in(0,t-\varepsilon)\), the gradient 
\[
\nabla_x G_{t-s}(x,y)
=-\frac{2\Xi(0)}{1-\rho_{t-s}^2} 
(x-\rho_{t-s}y) G_{t-s}(x,y),
\]
where \(\rho_{t-s}:=e^{-2\Xi(0)(t-s)}\),
possesses no singularity, and the right-hand side is integrable in
\((s,y)\in(0,t-\varepsilon)\times\mathbb R^d\). Hence dominated
convergence gives
\[
\nabla_x I_{\varepsilon,R}(t,x)
\longrightarrow
\nabla_x I_\varepsilon(t,x)
:=
\int_0^{t-\varepsilon}
\nabla_x S_{t-s}^*(\nabla\cdot H(s,\cdot))(x)\,ds,
\quad R\to\infty.
\]

It remains to let \(\varepsilon\to 0\). Defining \(f_\varepsilon(s):=\mathbbm{1}_{(0,t-\varepsilon)}(s)\nabla_x S_{t-s}^* (\nabla\cdot H(s,\cdot))(x)\) with \(s\in(0,t)\), we note that \(f_\varepsilon(s)\to f(s)\) with \(f(s):=\nabla_x S_{t-s}^* (\nabla\cdot H(s,\cdot))(x)\) as \(\varepsilon\to 0\). We now need an \(L^1(0,t)\)-majorant independent of \(\varepsilon\).  Given~\eqref{nablaGnablaH}, one can bound
\begin{multline}\label{fs_aux}
|f(s)|\lesssim 
\frac{c_2 \Xi(0)^{d/2+1}}{\min(1,\sqrt{s/2})(1-\rho_{t-s}^2)^{d/2+1}}
\int_{\R^d} \Vert x-\rho_{t-s} y\Vert_2 
(1+\Vert y\Vert_2)^{2} \\
\times
\exp\left\{
-\left(
\frac{\Xi(0)\Vert x-\rho_{t-s} y\Vert_2^2}{1-\rho_{t-s}^2}
+\frac{5}{4}\Xi(0)\Vert y\Vert_2^2
\right)
\right\}\,dy\\
\lesssim 
\frac{c_2 \Xi(0)^{d/2+1}}{\min(1,\sqrt{s/2})(1-\rho_{t-s}^2)^{d/2+1}}
\exp\left\{
-\frac{5\Xi(0)^2}{4\Xi(0)\rho_{t-s}^2+5\Xi(0)(1-\rho_{t-s}^2)}
\Vert x\Vert_2^2
\right\}\\
\times \int_{\R^d}
\Vert x-\rho_{t-s}y\Vert_2(1+\Vert y\Vert_2)^{2}
\exp\left\{
-A_{1,\Xi}
\left\Vert 
y-\frac{A_{2}}{A_{1,\Xi}}x
\right\Vert_2^2
\right\}
\,dy,
\end{multline}
where \(A_{1,\Xi}\) and \(A_2\) are defined as 
\begin{equation*}\label{A12}
A_{1,\Xi}:= \frac{4\Xi(0)\rho_{t-s}^2+5\Xi(0)(1-\rho_{t-s}^2)}{4(1-\rho_{t-s}^2)}
\quad\text{and}\quad
A_{2}:=\frac{\Xi(0)\rho_{t-s}}{1-\rho_{t-s}^2}.
\end{equation*}
Making a change of variables \(z=y-A_{2}x/A_{1,\Xi}\) and observing that
\[
\Vert x-\rho_{t-s}y\Vert_2
\leq \frac{5\Xi(0)}{4A_{1,\Xi}}\Vert x\Vert_2
+\rho_{t-s}\Vert z\Vert_2
\]
and
\[
1+\Vert y\Vert_2 
\leq 1+\Vert z\Vert_2 + \frac{A_{2}}{A_{1,\Xi}}\Vert x\Vert_2
\leq \left(1+\frac{A_{2}}{A_{1,\Xi,}}\Vert x\Vert_2\right)
(1+\Vert z\Vert_2),
\]
we then bound the space-integral by
\begin{multline}\label{aux_fs}
	\int_{\R^d}
\Vert x-\rho_{t-s}y\Vert_2(1+\Vert y\Vert_2)^{2}
\exp\left\{
-A_{1,\Xi}
\left\Vert 
y-\frac{A_{2}}{A_{1,\Xi}}x
\right\Vert_2^2
\right\}
\,dy\\
\leq 
\left(
1+\frac{A_{2}}{A_{1,\Xi}}\Vert x\Vert_2
\right)^{2}
\int_{\R^d}
\left(
\frac{5\Xi(0)}{4A_{1,\Xi}}\Vert x\Vert_2 
+\rho_{t-s} \Vert z\Vert_2
\right) (1+\Vert z\Vert_2)^{2}
\exp\{-A_{1,\Xi}\Vert z\Vert_2^2\}
\,dz\\
\lesssim 
A_{1,\Xi}^{-(d+1)/2}
\left(1+A_{1,\Xi}^{-1}\right)
\left(
\frac{5\Xi(0)}{4\sqrt{A_{1,\Xi}}}
\Vert x\Vert_2
+\rho_{t-s})
\right)\left(
1+\frac{A_{2}}{A_{1,\Xi}}\Vert x\Vert_2
\right)^{2}.
\end{multline}
Observing that 
\[
A_{1,\Xi}
\geq \frac{\min(\Xi(0),5\Xi(0)/4)}{1-\rho_{t-s}^2}
=\frac{\Xi(0)}{1-\rho_{t-s}^2},
\quad \quad 
\frac{A_{2}}{A_{1,\Xi}}
\leq \frac{\Xi(0)}{\min(\Xi(0),5\Xi(0)/4)}
=1,
\]
and 
\[
\frac{5\Xi(0)^2}{4\Xi(0)\rho_{t-s}^2+5\Xi(0)(1-\rho_{t-s}^2)} 
\geq \frac{5\Xi(0)^2}{\max\left(
4\Xi(0), 5\Xi(0)
\right)}
=\Xi(0),
\]
and substituting these into~\eqref{aux_fs} and then~\eqref{fs_aux}, we obtain
\begin{multline}\label{fs_aux2}
|f(s)|
\lesssim \frac{c_2 \Xi(0)^{d/2+1}}{\min(\Xi(0),1)^{(d+1)/2+3}}	
(1+\Xi(0)\Vert x\Vert_2)^{2}\exp\left\{
-\Xi(0)\Vert x\Vert_2^2
\right\}\\
\times 
\frac{1}{\min(1,\sqrt{s/2})} 
\left(
\Xi(0)\Vert x\Vert_2
+\frac{\rho_{t-s}}{\sqrt{1-\rho_{t-s}^2}}
\right),
\end{multline}
which is integrable near \(t\) by integrability of \((t-s)^{-1/2}\) and near zero due to\newline \(\int_0^t (\sqrt{s}\sqrt{t-s})^{-1}\,ds=\pi\). 
Hence, by the dominated convergence 
\[
\nabla_x I_\varepsilon(t,x)
\rightarrow \int_0 ^t \nabla_x S_{t-s}^* (\nabla\cdot H(s,\cdot))(x)\,ds,
\quad \varepsilon\to 0,
\]
i.e., 
\begin{equation}
\nabla_x p_t(x)
=\nabla_x S_t^* p_0(x) 
+\int_0^t \nabla_x S_{t-s}^*(\nabla\cdot H(s,\cdot))(x)\,ds.
\end{equation}

We should now bound the two terms above. For the first term, using Assumption~\ref{ass:main} and the explicit formula for the \(G_{t-s}(x,y)\), we get
\begin{multline*}
\nabla_x S_t^* p_0(x)
\lesssim C_0 \Xi(0)^{d/2+1}
\int_{\R^d}
\frac{\left\Vert x-\rho_t y\right\Vert_2}{(1-\rho_t^2)^{d/2+1}}
\exp\left\{
-\frac{\Xi(0)}{1-\rho_t^2}
\left\Vert x-\rho_t y\right\Vert_2^2
-b_0\Vert y\Vert_2^2
\right\}
\,dy\\
=C_0 \Xi(0)^{d/2+1}
\exp\left\{
-\frac{\Xi(0)b_0}{\Xi(0) \rho_t^2+b_0(1-\rho_t^2)}\Vert x\Vert_2^2
\right\}\\
\times
\int_{\R^d}
\frac{\left\Vert x-\rho_t y\right\Vert_2}{(1-\rho_t^2)^{d/2+1}}
\exp\left\{
-\frac{\sigma_t}{1-\rho_t^2}\Vert y-\mu\Vert_2^2
\right\}
\,dy,
\end{multline*}
where \(\sigma_t:=\Xi(0)\rho_t^2+b_0(1-\rho_t^2)\) and \(\mu:=\Xi(0)\rho_t x/\sigma_t\). Defining \(Z\) a normally distributed random variable with mean \(\mu\) and covariance matrix \((1-\rho_t^2)/(2\sigma_t)\), we then get
\begin{multline*}
	\int_{\R^d}
\frac{\left\Vert x-\rho_t y\right\Vert_2}{(1-\rho_t^2)^{d/2+1}}
\exp\left\{
-\frac{\sigma_t}{1-\rho_t^2}\Vert y-\mu\Vert_2^2
\right\}
\,dy
\lesssim 
\frac{1}{(1-\rho_t^2)\sigma_t^{d/2}}
\E\left[
\Vert x-\rho_t Z\Vert_2
\right]\\
\lesssim \frac{1}{(1-\rho_t^2)\sigma_t^{d/2}}
\left(
\Vert x-\mu \rho_t\Vert_2 + \Vert Z-\mu\Vert_2
\right)\\
\hspace{0.4cm}\lesssim \frac{1}{(1-\rho_t^2)\sigma_t^{d/2}}
\left(
\sqrt{\frac{1-\rho_t^2}{\sigma_t}}\Vert x\Vert_2
+\sqrt{\frac{1-\rho_t^2}{\sigma_t}}
\right)\\
=\frac{1}{\sigma_t^{(d+1)/2}\sqrt{1-\rho_t^2}}\left(1+\Vert x\Vert_2
\right)
\lesssim \frac{1+\Vert x\Vert_2}{\sigma_t^{(d+1)/2}\sqrt{1-\rho_t^2}},
\end{multline*}
from which
\begin{multline}\label{reg0}
\nabla_x S_t^* p_0(x)
\lesssim \frac{C_0 \Xi(0)(1+\Vert x\Vert_2)}{\left(\Xi(0) e^{-4\Xi(0) t}+b_0(1-e^{-4\Xi(0) t})\right)^{(d+1)/2}}\\
\times \frac{1}{\sqrt{1-e^{-4\Xi(0) t}}}
\exp\left\{
-\frac{\Xi(0)b_0}{\Xi(0) e^{-4\Xi(0) t}+b_0(1-e^{-4\Xi(0)t})}\Vert x\Vert_2^2
\right\}.
\end{multline}
To bound the second term, we consider separately two regimes: \(t-s\leq 1\) and \(t-s>1\). In the first regime, one can employ~\eqref{fs_aux2} to get that, since 
\[
\int_{\max(0,t-1)}^t 
\frac{1}{\min(1,\sqrt{s/2})}
\,ds
=2\sqrt{2t}\mathbbm{1}_{t\leq 1}
+(\sqrt{t}-\sqrt{t-1})\mathbbm{1}_{t>1}
\lesssim 1
\]
and
\begin{multline*}
\int_{\max(0,t-1)}^t 
\frac{\rho_{t-s}}{\min(1,\sqrt{s/2})\sqrt{(1-\rho_{t-s}^2)}}
\,ds
=\sqrt{2}\int_{\max(0,t-1)}^t 
\frac{\rho_{t-s}}{\sqrt{s(1-\rho_{t-s}^2)}}
\,ds\\
\leq 
\sqrt{2}\int_{\max(0,t-1)}^t \frac{1}{\sqrt{t-r}\sqrt{1-e^{-2\Xi(0) r}}}
\,dr
\leq \frac{\sqrt{2}}{\sqrt{1-e^{-2\Xi(0)}}}\int_0^1 \frac{1}{\sqrt{(t-r)r}}\,dr\\
=\frac{\sqrt{2}\pi}{\sqrt{1-e^{-2\Xi(0)}}},
\quad t\leq 1,
\end{multline*}
whereas 
\begin{multline*}
\int_{\max(0,t-1)}^t 
\frac{\rho_{t-s}}{\min(1,\sqrt{s/2})\sqrt{(1-\rho_{t-s}^2)}}
\,ds
\leq \int_{0}^1
\frac{1}{\sqrt{1-e^{-2\Xi(0)r}}}\,dr\\
\leq \frac{1}{\sqrt{1-e^{-2\Xi(0)}}}
\int_0^1 r^{-1/2}\,dr
=\frac{2}{\sqrt{1-e^{-2\Xi(0)}}},
\end{multline*}
it follows that
\begin{multline}\label{reg1}
\int_{\max(0,t-1)}^t \nabla_x S_{t-s}^*(\nabla\cdot H(s,\cdot))(x)\,ds
\lesssim 
\frac{c_2 \Xi(0)^{d/2+1}}{\min(\Xi(0),1)^{(d+1)/2+3}}	
(1+\Xi(0)\Vert x\Vert_2)^{2}\\
\times 
\left(
\Xi(0)\Vert x\Vert_2
+\frac{1}{\sqrt{1-e^{-2\Xi(0)}}}
\right)
\exp\left\{
-\Xi(0)
\Vert x\Vert_2^2
\right\}.
\end{multline}

Since on the region \(t-s>1\) the estimate~\eqref{nablaGnablaH} would produce a non-decaying term in \(t\), we use the divergence structure there instead. Namely, we note that, since
\[
S_{t-s}^*(\nabla_y\cdot H(s,\cdot))(x)
=
-\int_{\mathbb R^d}
\nabla_yG_{t-s}(x,y)\cdot H(s,y)\,dy,
\]
it holds that
\[
\nabla_x S_{t-s}^*(\nabla_y\cdot H(s,\cdot))(x)
=
-\int_{\mathbb R^d}
\nabla_x\nabla_yG_{t-s}(x,y)\,H(s,y)\,dy,
\]
from which
\[
\left\Vert 
\nabla_x S_{t-s}^*(\nabla_y\cdot H(s,\cdot))(x)
\right\Vert_2
\le
\int_{\mathbb R^d}
\left\Vert\nabla_x\nabla_yG_{t-s}(x,y)\right\Vert
\left\Vert H(s,y)\right\Vert_2\,dy .
\]
Given that
\[
\nabla_y G_{t-s}(x,y)
=
\frac{2\Xi(0)\rho_{t-s}}{1-\rho_{t-s}^2}
(x-\rho_{t-s}y)G_{t-s}(x,y),
\]
we have
\begin{multline*}
\left\|\nabla_x\nabla_yG_{t-s}(x,y)\right\|_{\mathrm{op}}
\lesssim
\frac{\Xi(0)^{d/2+1}\rho_{t-s}}
{(1-\rho_{t-s}^2)^{d/2+1}}
\left(
1+
\frac{\Xi(0)\Vert x-\rho_{t-s}y\Vert_2^2}
{1-\rho_{t-s}^2}
\right)\\
\times
\exp\left\{
-\frac{\Xi(0)\Vert x-\rho_{t-s}y\Vert_2^2}
{1-\rho_{t-s}^2}
\right\}.
\end{multline*}
Using again that by~\eqref{gaussbound_pt_unif}
\[
\Vert H(s,y)\Vert_2
\lesssim c^2 L_{\mathrm{Lip}}(\Xi)
\Vert y\Vert_2\exp\{-2\Xi(0)\Vert y\Vert_2^2\},
\]
we get, by the exact same way as before, 
\begin{multline*}
\int_{\R^d} 
\left(
1+
\frac{\Xi(0)\Vert x-\rho_{t-s}y\Vert_2^2}
{1-\rho_{t-s}^2}
\right)
\Vert y\Vert_2\exp\left\{
-\frac{\Xi(0)\Vert x-\rho_{t-s}y\Vert_2^2}
{1-\rho_{t-s}^2}
-2\Xi(0)\Vert y\Vert_2^2
\right\}
\,dy\\
\lesssim 
(1-\rho_{t-s}^2)^{d/2}
\frac{(1+\Xi^2(0)\Vert x\Vert_2^2)(1+\Xi(0)\Vert x\Vert_2)}{\min(\Xi(0),1)^{(d+5)/2}}
\exp\left\{
-\frac{2\Xi(0)^2}{\max(\Xi(0),2\Xi(0))}\Vert x\Vert_2^2
\right\},
\end{multline*}
from which 
\begin{multline*}
\int_0 ^{\max(0,t-1)} \nabla_x S_{t-s}^*(\nabla_y\cdot H(s,\cdot))(x)\,ds\\
\lesssim \Xi(0)^{d/2+1}c^2 L_{\mathrm{Lip}}(\Xi)
\frac{(1+\Xi^2(0)\Vert x\Vert_2^2)(1+\Xi(0)\Vert x\Vert_2)}{\min(\Xi(0),1)^{(d+5)/2}}\\
\times \exp\left\{
-\Xi(0)\Vert x\Vert_2^2
\right\}
\int_0^{\max(0,t-1)} \frac{\rho_{t-s}}{1-\rho_{t-s}^2}\,ds.
\end{multline*}
In the region \(t-s>1\), the denominator is bounded from below by \(1-\rho_{t-s}^2\geq 1-e^{-4\Xi(0)}>0\), yielding that
\begin{multline*}
\int_0^{\max(0,t-1)} \frac{\rho_{t-s}}{1-\rho_{t-s}^2}\,ds
\leq 
\frac{1}{1-e^{-4\Xi(0)}}
\int_0^t
e^{-2\Xi(0) s}\,ds\\
\leq \frac{1-e^{-2\Xi(0) t}}{2\Xi(0)(1-e^{-4\Xi(0)})}
\leq \frac{1}{2\Xi(0)(1-e^{-4\Xi(0)})}.
\end{multline*}
Combining the above with~\eqref{reg0} and~\eqref{reg1}, gathering also, with some abuse of notation, all the constants under a generic constant \(c_2(d,\kappa,K,b_0,c_0,\Xi(0),L_{\mathrm{Lip}}(\Xi))\), we obtain the result. 

\subsection{Proof of Lemma~\ref{momentcontrol}}
Applying the It\^o formula to \(f\colon (t,x)\mapsto \frac{1}{2}e^{2\varkappa t}\Vert x\Vert_2 ^2\), we get, for any \(t\geq 0\),
    \begin{multline*}
        \frac{1}{2}e^{2\varkappa t} \Vert X_t\Vert_2^2
        = \frac{1}{2}\Vert X_0\Vert_2^2
        +\varkappa \int_0 ^t e^{2\varkappa s}\Vert X_s\Vert_2^2\, ds
        -2\int_0 ^t e^{2\varkappa s} X_s \Xi(p_s(X_s))\cdot X_s\,ds\\
        +\sqrt{2}\int_0 ^t e^{2\varkappa s} X_s \cdot dW_s
        +\frac{d(e^{2\varkappa t}-1)}{2\varkappa}.
    \end{multline*}
    Since \(\mathbb{E}[\int_0 ^t e^{4\varkappa s} \Vert X_s\Vert^2_2 \,ds]<\infty\), the process \(M_t:=\int_0 ^t e^{2\varkappa s} X_s \cdot dW_s\) is a square-integrable martingale, and one can apply the Burkholder-Davis-Gundy inequality to get, for any \(p\geq 2\) and any \(T>0\) fixed,
    \begin{eqnarray*}
        \mathbb{E}\left[ 
        \sup\limits_{0\leq t\leq T}
        |M_t|^{p/2}
        \right]
        &\leq& (pC_{\text{BDG}})^{p/4} \mathbb{E}\left[
        \langle M, M\rangle_T ^{p/4}
        \right]\\
        &=&(pC_{\text{BDG}})^{p/4} \mathbb{E}\left[
        \left(\int_0 ^T
        e^{4\varkappa s} \Vert X_s\Vert_2^2\,ds
        \right)^{p/4}
        \right]\\
        &\leq& (pC_{\text{BDG}})^{p/4} \mathbb{E}\left[
        \left(
        \sup\limits_{0\leq s\leq T} 
        e^{2\varkappa s} \Vert X_s\Vert_2^2 \int_0 ^T e^{2\varkappa s}\,ds
        \right)^{p/4}
        \right]\\
        &\leq& 
        \left(
        \frac{pC_{\text{BDG}}e^{2\varkappa T}}{2\varkappa}
        \right)^{p/4}
        \mathbb{E}\left[
        \left(
        \sup\limits_{0\leq s\leq T} 
        e^{2\varkappa s} \Vert X_s\Vert_2^2 \,ds
        \right)^{p/4}
        \right],
    \end{eqnarray*}
    where \(C_{\text{BDG}}=C_{\text{BDG}}(p)>0\) is some fixed constant.
    Now, by the relation between geometric and arithmetic means, it can be seen that
    \[
    \mathbb{E}\left[
        \left(
        \sup\limits_{0\leq s\leq T} 
        e^{2\varkappa s} \Vert X_s\Vert_2^2 
        \right)^{p/4}
        \right]
        \leq A+\frac{1}{A}\mathbb{E}\left[
        \left(
        \sup\limits_{0\leq s\leq T} 
        e^{2\varkappa s} \Vert X_s\Vert_2^2 \,ds
        \right)^{p/2}
        \right]
    \]
    for any \(A>0\). Hence, under the choice \(A:=3^{p/2-1}\cdot 2^{3p/4+1}\left(C_{\text{BDG}}pe^{2\varkappa T}/(2\varkappa)\right)\) it follows that
    \begin{multline}\label{moment_aux}
        \mathbb{E}\left[
        \left(
        \sup\limits_{0\leq t\leq T} 
        e^{2\varkappa t} \Vert X_t\Vert_2^2
        \right)^{p/2}
        \right]
        \leq
        3^{p/2-1} \mathbb{E}\left[
        \Vert X_0\Vert_2^p
        \right]
        +3^{p/2-1} 2^{3p/4}\mathbb{E}\left[
        \sup\limits_{0\leq t\leq T} 
        |M_t|^{p/2}
        \right]\\
        + 3^{p/2-1} \mathbb{E}\left[
        \left(
        \sup\limits_{0\leq t\leq T} 
        \int_0 ^t \left(
        2\varkappa \Vert X_s\Vert_2^2 
        -4\Xi(p_s(X_s))\Vert X_s\Vert_2 ^2 +2d
        \right)e^{2\varkappa s}\,ds
        \right)^{p/2}
        \right]\\
        \leq 3^{p/2-1} \mathbb{E}\left[
        \Vert X_0\Vert_2^p
        \right]
        + 3^{p-2} \cdot 2^{3p/4+1} \left(\frac{pC_{\text{BDG}}e^{2\varkappa T}}{2\varkappa}\right)^{p/2}
        +\frac{1}{2}\mathbb{E}\left[
        \left(
        \sup\limits_{0\leq t\leq T} 
        e^{2\varkappa t} \Vert X_t\Vert_2^2
        \right)^{p/2}
        \right]\\
        + 3^{p/2-1} \mathbb{E}\left[
        \left(
        \sup\limits_{0\leq t\leq T} 
        \int_0 ^t \left(
        2\varkappa \Vert X_s\Vert_2^2 
        -4\Xi(p_s(X_s))\Vert X_s\Vert_2 ^2 +2d
        \right)e^{2\varkappa s}\,ds
        \right)^{p/2}
        \right].
            \end{multline}
        Noting that by Assumption~\ref{ass:main} it holds that
        \begin{equation}\label{dissipativity}
        <-\Xi(p_t(x))\nabla \Phi(x),x>
        =-2\Xi(p_t(x))\Vert x\Vert_2^2
        \leq -2\varkappa \Vert x\Vert_2^2,
        \end{equation}
        we get that the last summand in~\eqref{moment_aux} is upper bounded by
        \begin{multline*}
        3^{p/2-1} \mathbb{E}\left[
        \left(
        \sup\limits_{0\leq t\leq T} 
        \int_0 ^t \left(
        2\varkappa \Vert X_s\Vert_2^2 
        -4\Xi(p_s(X_s))\Vert X_s\Vert_2 ^2 +d
        \right)e^{2\varkappa s}\,ds
        \right)^{p/2}
        \right]\\
        \leq 3^{p/2-1} 
        \sup\limits_{0\leq t\leq T}
        \left(
        \int_0 ^t de^{2\varkappa s}\,ds
        \right)^{p/2}
        =3^{p/2-1}\left(
        \frac{d}{2\varkappa} \left(
        e^{2\varkappa T}-1
        \right)
        \right)^{p/2}.
        \end{multline*}
        This implies that
        \begin{multline*}
            e^{-\varkappa T}\left(
            \mathbb{E}\left[
        \left(
        \sup\limits_{0\leq t\leq T} 
        e^{2\varkappa t} \Vert X_t\Vert_2^2
        \right)^{p/2}
        \right]
            \right)^{1/p}\\
            \leq 
            3^{1/2-2/p+1/p^2}
            \left(
            \mathbb{E}\left[
        \Vert X_0\Vert_2^p
        \right]
        \right)^{1/p}
        +3^{3(1/2-1/p)}\cdot 2^{3/2+1/p}
        \sqrt{\frac{C_{\text{BDG}}}{\varkappa}}
        \sqrt{p}
        +3^{1/2-2/p+1/p^2}
        \sqrt{\frac{d}{\varkappa}}\\
        \leq \left(
        \sqrt{3}
        \left(
            \mathbb{E}\left[
        \Vert X_0\Vert_2^p
        \right]
        \right)^{1/p}
        +6\sqrt{6} 
        \sqrt{\frac{C_{\text{BDG}}}{\varkappa}}
        +\sqrt{3} 
        \sqrt{\frac{d}{\varkappa}}
        \right)
        \sqrt{p}.
        \end{multline*}
        Finally, noting that
        \[
            e^{-\varkappa T} \left(
            \mathbb{E}\left[
        \left(
        \sup\limits_{0\leq t\leq T} 
        e^{2\varkappa t} \Vert X_t\Vert_2^2
        \right)^{p/2}
        \right]
            \right)^{1/p} 
            \geq \sup\limits_{0\leq t\leq T}
            \left(
            \mathbb{E}\left[
            \Vert 
            X_t
            \Vert_2 ^p
            \right]
            \right)^{1/p},
        \]
        and \(T>0\) is arbitrary,
        we conclude the claim.

\subsection{Proof of Lemma~\ref{lem:expbound_X}}
Let \(A_t(x):=\int_0^t \Vert x\Vert_2^2\,ds\) and \(V(x):=\exp\{a\Vert x\Vert_2^2\}\) with \(a>0\). It can be seen that the process \(\widetilde{X}_t:=(A_t(X),X_t^1,\dots, X_t^d)\) is a \((d+1)\)-dimensional semimartingale. Applying the It\^{o} formula to the function \(f(\widetilde{x})=f(a,x):=\exp\{\lambda A_t(x)\}V(x)\) with some \(\lambda>0\) we get
    \begin{multline*}
        e^{\lambda A_t(X)}V(X_t)
        =V(X_0)+\sqrt{2}\int_0^t 
        e^{\lambda A_s(X)} \nabla V(X_s)\cdot dW_s\\
        +\int_0^t e^{\lambda A_s(X)}\left(
        \lambda V(X_s)\Vert X_s\Vert_2^2
        -2X_s\Xi(p_s(X_s))\cdot \nabla V(X_s)
        +\Delta V(X_s)
        \right)\,ds.
    \end{multline*}
    Given that \(\nabla V(x)=2axV(x)\), \(\Delta V(x)=(2ad+4a^2\Vert x\Vert_2^2)V(x)\), and~\eqref{dissipativity} holds, it then follows that
    \begin{multline*}
      e^{\lambda A_t(X)}V(X_t)
      \leq V(X_0)+\sqrt{2}\int_0^t 
        e^{\lambda A_s(X)} \nabla V(X_s)\cdot dW_s\\
        +\int_0^t e^{\lambda A_s(X)}\left(
        2ad+(\lambda-\delta)\Vert X_s\Vert_2^2
        \right) V(X_s)\,ds,
    \end{multline*}
    where \(\delta:=4a(\kappa-a)\). Chossing \(a\in(0,\kappa)\) and \(\lambda\in(0,\delta]\), one can then bound
    \begin{equation}\label{expb_aux}
    e^{\lambda A_t(X)}V(X_t)
      \leq V(X_0)+\sqrt{2}\int_0^t 
        e^{\lambda A_s(X)} \nabla V(X_s)\cdot dW_s
        +2ad\int_0^t e^{\lambda A_s(X)} V(X_s)\,ds.
    \end{equation}
    Define now a sequence of stopping times \(\tau_n:=\inf\{t\geq 0\colon \Vert X_t\Vert_2\geq n\}\). Since the stopped process \(e^{\lambda A_{\min(s,\tau_n)}(X)} \nabla V(X_{\min(s,\tau_n)})\) is square-integrable, the integral with respect to Brownian motion defines a martingale, and taking expectations to the both sides of~\eqref{expb_aux} evaluated at \(\min(t,\tau_n)\) we get
    \[
    \E\left[ 
    e^{\lambda A_{\min(t,\tau_n)}(X)}V(X_{\min(t,\tau_n)})
    \right]
    \leq \E\left[V(X_0)\right]
    +2ad\int_0^t \E\left[ 
    e^{\lambda A_{\min(s,\tau_n)}(X)} V(X_{\min(s,\tau_n)})
    \right]\,ds.
    \]
    By Gr\"onwall's inequality, it follows that
    \[
    \E\left[ 
    e^{\lambda A_{\min(t,\tau_n)}(X)}V(X_{\min(t,\tau_n)})
    \right]
    \leq
    e^{2adT}\E\left[V(X_0)\right],
    \]
    and, since \(V(x)\geq 1\) for all \(x\in\R^d\), 
    \[
    \E\left[ 
    e^{\lambda A_{\min(t,\tau_n)}(X)}
    \right]
    \leq
    e^{2adT}\E\left[V(X_0)\right].
    \]
    Since \(A_{\min(t,\tau_n)}(X)\uparrow A_t\), we conclude by the monotone convergence that 
    \(
    \E\left[ 
    e^{\lambda A_t(X)}
    \right]
    \leq
    e^{2adT}\E\left[V(X_0)\right],
    \)
    and, since \(a=\kappa/2\) maximises the interval for \(\lambda\), get the claim.

        \section{Auxiliary results}\label{sec:aux}
\begin{lem}\label{lem:KL_lower}
    Let \(p\) and \(q\) be two probability densities on \(\R^d\), \(d\geq 1\), such that \(p \ll q\). Then
    \begin{equation}\label{KL32}
    KL(p\Vert q) \geq 
    \frac{3}{2} \int_{\R^d} 
    \frac{(q(x)-p(x))^2}{p(x)+2q(x)}\,dx.
    \end{equation}
\end{lem}
\begin{proof}
    Indeed, one can rewrite
    \begin{eqnarray*}
        KL(p\Vert q)
        = \int_{\R^d} p(x)
        \log\left(
        \frac{p(x)}{q(x)}
        \right)\,dx
        &=&-\int_{\R^d} p(x)
        \log\left(
        \frac{q(x)}{p(x)}
        \right)\,dx\\
        &=&
        -\int_{\R^d} p(x)
        \log\left(
        1+\frac{q(x)-p(x)}{p(x)}
        \right)\,dx
    \end{eqnarray*}
    and, applying the generic inequality
    \[
    \log(1+z)\leq z-\frac{z^2}{2+\frac{4}{3}z}, \quad z>-1,
    \]
    obtain
    \begin{multline*}
        KL(p\Vert q)
        \geq
        -\int_{\R^d} p(x)\frac{q(x)-p(x)}{p(x)}\,dx
        +\int_{\R^d} p(x)
        \left(
        \frac{q(x)-p(x)}{p(x)}
        \right)^2
        \frac{1}{2+\frac{4}{3}\frac{q(x)-p(x)}{p(x)}}\,dx\\
        \geq \frac{3}{2}\int_{\R^d} 
    \frac{(q(x)-p(x))^2}{p(x)+2q(x)}\,dx.
    \end{multline*}
\end{proof}

\begin{lem}\label{lem:heat_aux}
Let \(Q:=J\times O\) be a bounded cylinder in \(\R_+\times\R^d\), where \(J\subset\mathbb R_+\) is a bounded open interval
and \(O\subset\mathbb R^d\) is a ball. Suppose that
\[
u_t-\Delta u=\operatorname{div}G+H
\]
in \(Q\) in the sense of distributions, with \(G\in C^{\beta/2,\beta}(Q)\) for some \(\beta\in(0,1)\) and \(H\in L^\infty(Q)\). 
Then the gradient of \(u\) exists in a classical sense and satisfies
\[
\nabla_x u\in C^{\beta/2,\beta}_{\mathrm{loc}}(Q).
\]
\end{lem}
\begin{proof}[Proof of the lemma]
Let \(Q_r\Subset Q_{2r}\) be nested cylinders with radii \(r, 2r>0\) such that \(Q_{2r}\Subset Q\). By Theorem 2.1 in~\cite{dong2015schauder}, there exists a unique solution to the equation 
\[
\partial_t w_t-\Delta w=\operatorname{div}G+H
\]
such that \(w\in C^{(1+\beta)/2,1+\beta}(Q_{2r})\). Hence, \(u\) and \(w\) solve the same equation in \(Q_{2r}\). Thus, \(h:=u-w\) satisfies 
\[
\partial_t h_t-\Delta h=0
\]
in \(Q_{2r}\) in the sense of distributions. Since the heat operator is hypoelliptic, it follows that \(h\in C^\infty_{\mathrm{loc}}(Q_{2r})\). From this we conclude that
\[
u=w+h\in C^{(1+\beta)/2,1+\beta}_{\mathrm{loc}}(Q_{2r}),
\]
and, in particular,
\(
\nabla_x u\in C_{\mathrm{loc}}^{\beta/2,\beta}(Q_r).
\)
Since \(Q_r\Subset Q\) was arbitrary,
\(
\nabla_xu\in C^{\beta/2,\beta}_{\mathrm{loc}}(Q).
\)
\end{proof}

\begin{lem}
The density \(p_t\) can be represented as
\begin{equation}\label{p_duh} 
p_t(x)=S_t^* p_0(x) + \int_0^t S_{t-s}^* (\nabla\cdot H(s,\cdot))(x)\,ds
\end{equation}
in the sense of distributions, where \(S_t^* f(x):=\int_{\R^d} G_t(x,y)f(y)\,dy\) for a given function \(f\), 
\[
G_t(x,y)
=
\left(
\frac{a_0}
{\pi(1-e^{-4a_0t})}
\right)^{d/2}
\exp\left\{
-\frac{
a_0\Vert x-e^{-2a_0t}y\Vert^2
}{
1-e^{-4a_0t}
}
\right\}
\]
is the Ornstein-Uhlenbeck kernel with some fixed \(a_0>0\),
and \(H(t,x):=2x\bigl(\Xi(p_t(x))-a_0\bigr)p_t(x)\).
\end{lem}

\begin{proof}
For a fixed \(a_0>0\), rewrite~\eqref{eq:NFPE} as
\begin{equation}
\label{eq:split}	
\partial_t p=L_0^* p+\nabla\cdot H,
\end{equation}	
where \(L_0^*p:=\Delta p+\nabla\cdot(2a_0xp)\) is the adjoint operator to \(L_0\varphi=\Delta\varphi-2a_0x\cdot\nabla\varphi\).
Denote by \((S_t)_{t\geq 0}\) the Markov semigroup generated by \(L_0\), i.e., 
\[
S_t\varphi(y)
=
\int_{\mathbb R^d}
G_t(x,y)\varphi(x)\,dx.
\]
For any \(0<\tau<t<T\) and any \(\varphi\in C_0^\infty\), it then holds that
\begin{equation*}\label{duham}
\int_{\R^d}
p_t(x)\varphi(x)\,dx
=\int_{\R^d} p_\tau(x)S_{t-\tau} \varphi(x)\,dx
-\int_\tau^t \int_{\R^d} 
H(s,x)\cdot \nabla S_{t-s}\varphi(x)\,dx\,ds.
\end{equation*}
Indeed, fix some \(t>0\) and define
\[
h(s):=\int_{\R^d} p_s(x)S_{t-s}\varphi(x)\,dx,
\quad s\in(\tau,t).
\]
Since under the Ornstein-Uhlenbeck semigroup \(S_{t-s}\varphi\), \(\nabla S_{t-s}\varphi\) and \(\Delta S_{t-s}\varphi\) are all bounded, \(L_0 S_{t-s}\varphi\) has at most linear growth, and since by Theorem~\ref{thm:wellposednessMV} the density \(p_s\) has uniformly Gaussian tails, \(p_sL_0 S_{t-s}\varphi\in L^1(\R^d)\). Hence, one can differentiate \(h\) under the integral sign, which, observing also that \(\partial_s S_{t-s}\varphi=-L_0 S_{t-s}\varphi\), yields
\begin{multline*}
\frac{d}{ds}h(s)
=\int_{\R^d} (\partial_s p_s(x))S_{t-s}	\varphi(x)\,dx
-\int_{\R^d} p_s(x) L_0 S_{t-s}\varphi(x)\,dx\\
=\int_{\R^d} L_0^* p_s(x) S_{t-s}\varphi(x)\,dx
+\int_{\R^d} (\nabla\cdot H(s,x))S_{t-s}\varphi(x)\,dx
-\int_{\R^d} p_s(x)L_0 S_{t-s}\varphi(x)\,dx\\
=\int_{\R^d} (\nabla\cdot H(s,x))S_{t-s}\varphi(x)\,dx,
\end{multline*}
where the last equality follows from the duality. Since \(S_{t-s}\) is bounded, and by~\eqref{gaussbound_pt} and~\eqref{gaussbound_grad} both \(p_t\) and \(\nabla p_t\) have Gaussian decay, we then have
\[
\frac{d}{ds}h(s)
=\int_{\R^d} (\nabla\cdot H(s,x))S_{t-s}\varphi(x)\,dx
=-\int_{\R^d} H(s,x)\cdot\nabla S_{t-s} \varphi(x)\,dx.
\]
Integrating the last equality over \((\tau,t)\), we then get
\begin{multline*}
h(t)-h(\tau)
=-\int_\tau ^t H(s,x)\cdot \nabla S_{t-s}\varphi(x)\,dx\,ds\\
\Leftrightarrow
\int_{\R^d} p_t(x)\varphi(x)\,dx
=\int_{\R^d} p_\tau(x)S_{t-\tau}\varphi(x)\,dx
-\int_\tau^t \int_{\R^d}	 H(s,x)\cdot\nabla S_{t-s}\varphi(x)\,dx\,ds.
\end{multline*}
Now, since by Theorem~\ref{thm:wellposednessMV} the process \(X\) is a strong solution, it has continuous paths, and \(X_\tau\to X_0\) almost surely as \(\tau\to 0\). Hence, \(S_{t-\tau}\varphi(X_\tau)\to S_t\varphi(X_0)\) a.s.\ as \(\tau\to 0\), and, since \(|S_t\varphi(x)|\leq\Vert \varphi\Vert_\infty\), by the dominated convergence it holds that \(\E[S_{t-\tau}\varphi(X_\tau)]\to \E[S_t\varphi(X_0)]\), i.e.,
\[
\int_{\R^d} p_\tau (x) S_{t-\tau}\varphi(x)\,dx
\to \int_{\R^d} p_0(x) S_t\varphi(x)\,dx,
\]
which yields~\eqref{p_duh}.
\end{proof}
        

\begin{thm}[Girsanov's theorem; see, e.g.,~\cite{KS-BM}]\label{girsanov_thm}
    Let \((\Omega, \{\F_t\}_{t\geq 0},\PP)\) be a probability space, and let \(\{\beta_t\}_{t\geq 0}\), be an \(\F_t\)-adapted stochastic process with values in \(\R^d\) such that
    \[
    \PP\left(
    \int_0 ^T \Vert \beta_s\Vert_2^2\,ds
    <\infty
    \right)
    =1, 
    \quad \forall\, 0<T<\infty.
    \]
    Then, if the process
    \[
    \mathcal{E}_t:=\exp\left\{
    \int_0 ^t \beta_s\,dW_s
    -\frac{1}{2}\int_0 ^t \Vert \beta_s\Vert_2^2\,ds
    \right\},
    \quad t\geq 0,
    \]
    is an \(\F_t\)-martingale under \(\PP\), then for any fixed \(0<T<\infty\) the process
    \[
    \widetilde{W}_t:= -\int_0 ^t \beta_s\,ds+ W_t, 
    \quad 0\leq t\leq T,
    \]
    is an \(\F_t\)-Brownian motion on \((\Omega,\{\F_t\}_{0\leq t\leq T},\widetilde{\PP})\) with \(\widetilde{\PP}\) a probability measure such that \(\frac{d\widetilde{\PP}}{d\PP}\Bigr|_{\F_T}=\mathcal{E}_T\).
\end{thm}

\begin{lem}\label{OU-expmoments}
    Let \(Y_t\) be the Ornstein-Uhlenbeck process defined in~\eqref{OU_a0}, and let \(c>0\) be a fixed constant. Then for any \(0<T<\infty\) fixed it holds that
    \begin{multline*}
        \E_{\PP_y}\left[ 
        \exp\left\{
        c\int_0 ^T \Vert Y_s\Vert_2^2
        \,ds
        \right\}
        \right]
        = 
        \left(
        \frac{\sqrt{a_0^2-c}}{a_0\sinh(2\sqrt{a_0^2-
        c}T)+\sqrt{a_0^2-c}\cosh(2\sqrt{a_0^2-c}T)}
        \right)^{d/2}
        \\\exp\left\{
        a_0 dT 
        +\frac{c(1-e^{-4\sqrt{a_0^2-c}T})}{
        2\left(
        (\sqrt{a_0^2-c}-a_0)e^{-4\sqrt{a_0^2-c}T} + \sqrt{a_0^2-c}+a_0
        \right)
        }
        \Vert y\Vert_2^2
        \right\},
    \end{multline*}
    provided that \(c<a_0^2\).
\end{lem}
\begin{proof}
Let us start with noting that~\eqref{OU_a0} admits a unique solution
            \[
            Y_t
            =Y_0 e^{-2a_0 t}
            +\sqrt{2}\int_0 ^t 
            e^{-2a_0(t-s)}\,dW_s,
            \quad t\geq 0,
            \]
            with mean \(\mu^{OU}(t):=\E[Y_t]=e^{-2a_0 t}\E[Y_0]\) and covariance function
            \begin{equation}\label{OU_cov}
            K^{OU}(s,t)=\frac{1}{2a_0}\left(
            e^{-2a_0|t-s|}
            -e^{-2a_0(t+s)}
            \right).
            \end{equation}
            Since conditional on \(Y_0=y\) the process \(Y_t\) can then be decomposed as \(Y_t=\mu^{OU}(t)+Z_t\), where \(Z_t\sim\mathcal{N}(0_d, K^{OU}(t,t) I_d)\) with \(0_d\) a \(d\)-dimensional vector of zeros, one can rewrite the expectation as
            \begin{multline*}
            \E_{\PP^{OU}_y}\left[
                \exp\left\{
                c\int_0 ^T \Vert Y_s\Vert_2^2\,ds
                \right\}\right]
                =\exp\left\{
                c\int_0 ^T\Vert \mu^{OU}(s) \Vert_2^2\,ds
                \right\}\\
                \times
                \E_{\PP^{OU}_y}\left[
                \exp\left\{
                2c\int_0 ^T 
                \mu^{OU}(s) \cdot Z_s\,ds
                +c\int_0 ^T 
                \Vert Z_s\Vert_2^2
                \right\}\right].
            \end{multline*}
            Let us now consider the operator 
            \begin{equation}\label{K-OU-operator}
            (K^{OU} f)(s)
            := \int_0 ^T K^{OU}(s,t) f(t)\,dt,
            \quad f\in L^2([0,T]), 
            \quad \forall 0\leq s\leq T.
            \end{equation}
            By the Mercer's theorem this operator has eigenfunctions \(\phi_n(s)\) forming an orthonormal basis in \(L^2([0,T])\) and nonnegative eigenvalues \(\lambda_n\) such that
            \[
            \int_0 ^T K^{OU}(s,t)
            \phi_n(t)\,dt
            =\lambda_n \phi_n(s),\quad
            \forall n\geq 1.
            \]
            In addition, due to the Karhunen-Lo\`eve theorem, one can decompose
            \[
            Z_s=\sum_{n=1}^\infty \phi_n(s)\xi_n,
            \]
            where \(\{\xi_n\}_{n\geq 1}\sim \text{i.i.d.} \mathcal{N}(0_d, \lambda_n I_d)\). Hence, defining also \(\mu_n:=\int_0 ^T\phi_n(s)\mu^{OU}(s)\,ds\), we can rewrite
            \begin{multline*}
                \exp\left\{
                c\int_0 ^T\Vert \mu^{OU}(s) \Vert_2^2\,ds
                \right\}
                \E_{\PP^{OU}_y}\left[
                \exp\left\{
                2c\int_0 ^T 
                \mu^{OU}(s) \cdot Z_s\,ds
                +c\int_0 ^T 
                \Vert Z_s\Vert_2^2
                \,ds
                \right\}\right]\\
                =\exp\left\{
                c\int_0 ^T\Vert \mu^{OU}(s) \Vert_2^2\,ds
                \right\}
                \E_{\PP^{OU}_y}\left[
                \exp\left\{
                2c\sum_{n=1}^\infty \mu_n\cdot\xi_n
                +c\int_0 ^T \left\Vert\sum_{n=1}^\infty \phi_n(s)\xi_n \right\Vert_2^2\,ds
                \right\}
                \right]\\
                =\exp\left\{
                c\int_0 ^T\Vert \mu^{OU}(s) \Vert_2^2\,ds
                \right\}
                \prod_{n=1}^\infty 
                \E_{\PP^{OU}_y}\left[ 
                e^{2c\mu_n\xi_n+c\Vert\xi_n\Vert_2^2}
                \right].
            \end{multline*}
            
            Since for each \(n\geq 1\)
            \begin{eqnarray*}
                \E_{\PP^{OU}_y}\left[ 
                e^{2c\mu_n\xi_n+c\Vert\xi_n\Vert_2^2}
                \right]
                &=&\frac{1}{(2\pi\lambda_n)^{d/2}}
                \int_{\R^d} e^{2c\mu_n z+c\Vert z\Vert_2^2 - \frac{\Vert z\Vert_2^2}{2\lambda_n}}\,dz\\
                &=&\frac{1}{(2\pi\lambda_n)^{d/2}}
                \int_{\R^d}e^{-\frac{1}{2\lambda_n}(1-2c\lambda_n)\Vert z\Vert_2^2+2c\mu_n z}\,dz\\
                &=&\frac{e^{\frac{2c^2\Vert \mu_n\Vert_2^2\lambda_n}{1-2c\lambda_n}}}{(2\pi\lambda_n)^{d/2}}
                \int_{\R^d} e^{-\frac{1-2c\lambda_n}{2\lambda_n} \Vert z-\frac{2c\mu_n\lambda_n}{1-2c\lambda_n}\Vert_2^2}\,dz\\
                &=& (1-2c\lambda_n)^{-d/2} e^{\frac{2c^2\Vert \mu_n\Vert_2^2\lambda_n}{1-2c\lambda_n}},
            \end{eqnarray*}
            provided that \(c<(2\lambda_n)^{-1}\) for all \(n\geq 1\), we further get
            \begin{multline}\label{aux_OU_infprod}
                \exp\left\{
                c\int_0 ^T\Vert \mu^{OU}(s) \Vert_2^2\,ds
                \right\}
                \prod_{n=1}^\infty 
                \E_{\PP^{OU}_y}\left[ 
            e^{2c\mu_n\xi_n+c\Vert\xi_n\Vert_2^2}
                \right]\\
                = \exp\left\{
                c\int_0 ^T\Vert \mu^{OU}(s) \Vert_2^2\,ds
                \right\}
                \prod_{n=1}^\infty
                (1-2c\lambda_n)^{-d/2} e^{\frac{2c^2\Vert \mu_n\Vert_2^2\lambda_n}{1-2c\lambda_n}}\\
                =\left(
                \prod_{n=1}^\infty (1-2c\lambda_n)^{-d/2}
                \right)
                \exp\left\{
                \sum_{n=1}^\infty 
                \frac{c\Vert \mu_n\Vert_2^2}{(1-2c\lambda_n)}
                \right\},
                \quad c<(2\max_n \lambda_n)^{-1},
            \end{multline}
            where the last equality follows from the observation that \(\int_0 ^T \Vert \mu^{OU}(s)\Vert_2^2\,ds=\sum_{n=1}^\infty \Vert \mu_n\Vert_2^2\). Given that by~\eqref{lambda_bound} the eigenvalues \(\lambda_n\) satisfy \(\lambda_n^{-1}<2a_0^2\), the sufficient condition for \(c\) is given by \(c<a_0^2\).
            
It now remains to evaluate the exponential and infinite product terms. The rest of the proof is thus divided into two steps.
            
            \underline{Step 1.} Let us start with the exponential term. Using again the very definitions of \(\mu_n\) and \(\mu^{OU}\), one can rewrite
            \[
            \Vert \mu_n\Vert_2^2
            =\left(
            \int_0 ^T 
            \phi_n(s)\mu^{OU}(s)\,ds
            \right)^2
            =\Vert y\Vert_2^2
            \left(
            \int_0 ^T 
            \phi_n(s)e^{-2a_0 s}\,ds
            \right)^2.
            \]
            Denote for convenience \(g(s):=e^{-2a_0 s}\), \(0\leq s\leq T\). Since \(\{\phi_n\}_{n\geq 1}\) is the orthonormal basis in \(L^2([0,T])\), one can decompose \(g\) as \(g=\sum_{n\geq 1} \langle g,\phi_n\rangle \phi_n\) with \(\langle g,\phi_n\rangle=\int_0 ^T g(s)\phi_n(s)\,ds\), and, since \(c<(2\max_n \lambda_n)^{-1}\), observe that
            \[
            (I-2cK^{OU})^{-1} g
            =\sum_{n=1}^{\infty} 
            \frac{1}{1-2c\lambda_n} 
            \langle g,\phi_n\rangle
            \phi_n,
            \]
            where \(I\) is the identity operator. Taking the scalar product with \(g\) on both sides, we further get
            \[
            \langle g, (I-2cK^{OU})^{-1} g \rangle
            =\sum_{n=1}^{\infty} 
            \frac{1}{1-2c\lambda_n} 
            \langle g,\phi_n\rangle^2,
            \]
            which yields 
            \[
            \sum_{n=1}^\infty 
            \frac{c\Vert \mu_n\Vert_2^2}{(1-2c\lambda_n)}
            =c\Vert y\Vert_2^2
            \sum_{n=1}^\infty
            \frac{\langle g,\phi_n\rangle^2}{(1-2c\lambda_n)}
            =c\Vert y\Vert_2^2
            \langle g, (I-2cK^{OU})^{-1} g \rangle.
            \]
            Hence, it suffices to compute the scalar product to the right-hand side. 

            Denote further \(u:=(I-2cK^{OU})^{-1} g\). Then it holds that \(u(I-2cK^{OU})= g\), i.e.,
            \begin{equation}\label{aux_u}
            u(s)
            -2c\int_0^T K^{OU}(s,t)u(t)\,dt
            =g(s),
            \quad \forall 0\leq s\leq T.
            \end{equation}
            It can be noticed that \(K^{OU}\) satisfies 
            \[
            \left(
            \frac{d^2}{ds^2}
            -4a_0^2
            \right)
            K^{OU}(s,t)
            =-2\delta(t-s)
            \]
            with boundary conditions \(K^{OU}(0,T)=0\) and \(\partial_s K^{OU}(T,t)=-2a_0 K^{OU}(T,t)\) (for more detail, see Lemma~\ref{eigenvalues_OU}). Hence, applying the operator \(\mathcal{D}_s:=\frac{d^2}{ds^2}-4a_0^2\) to~\eqref{aux_u}, we get
            \begin{multline*}
                u''(s)-4a_0^2 u(s)
                -2c\int_0^T 
                \mathcal{D}_s 
                K^{OU}(s,t)
                u(t)\,dt
                =g''(s)-4a_0^2 g(s)\\
                \Leftrightarrow 
                u''(s)-4(a_0^2-c) u(s) 
                =0,
            \end{multline*}
            and, since \(b:=2\sqrt{a_0^2-c}>0\), the equation above admits the unique solution \(u(s):=Ae^{bs}+Be^{-bs}\) with some \(A,B\in\R\). Given that \(u(0)=1\), it follows that \(A+B=1\). Furthermore, using the boundary condition for \(K^{OU}\) and~\eqref{aux_u} one more time, we get
            \begin{multline*}
                u'(T)
                -2c\int_0 ^T \partial_s
                K^{OU}(T,t) 
                u(t)
                \,dt
                =-2a_0 e^{-2a_0 T} \\
                \Leftrightarrow
                u'(T)
                +4a_0 c 
                \int_0 ^T
                K^{OU}(T,t)
                u(t)\,dt
                =-2a_0 e^{-2a_0 T}\\
                \Leftrightarrow 
                u'(T)
                +2a_0 \left(u(T)-e^{-2a_0 T}\right)
                =-2a_0 e^{-2a_0 T}\\
                \Leftrightarrow 
                u'(T)+2a_0 u(T)=0.
            \end{multline*}
The latter relation implies that
\[
    bAe^{bT} - bBe^{-bT}
    +2a_0 Ae^{bT} + 2a_0 B e^{-bT}
    =0 
    \qquad
    \Leftrightarrow A=\frac{b-2a_0}{b+2a_0}Be^{-2bT},
\]
from which 
\[
B=\frac{b+2a_0}{(b-2a_0)e^{-2bT} + (b+2a_0)}
\quad \text{and} \quad
A=\frac{(b-2a_0)e^{-2bT}}{(b-2a_0)e^{-2bT} + (b+2a_0)}.
\]
This leads to 
\begin{multline*}
    \langle g, u\rangle
    =\int_0 ^T e^{-2a_0 s}\left(
    Ae^{bT} + Be^{-bs}
    \right)
    \,ds
    =\frac{A(1-e^{-(2a_0-b)T})}{2a_0-b}
    +\frac{B(1-e^{-(2a_0+b)T})}{2a_0+b}\\
    =\frac{1-e^{-2bT}}{(b-2a_0)e^{-2bT} + b+2a_0},
\end{multline*}
which finally yields
\begin{multline*}
\sum_{n=1}^\infty 
\frac{c\Vert \mu_n\Vert_2^2}{1-2c\lambda_n}
=c\Vert y\Vert_2^2
\frac{1-e^{-2bT}}{(b-2a_0)e^{-2bT} + b+2a_0}\\
=\frac{c(1-e^{-4\sqrt{a_0^2-c}T})}{
2\left(
(\sqrt{a_0^2-c}-a_0)e^{-4\sqrt{a_0^2-c}T} + \sqrt{a_0^2-c}+a_0
\right)
}
\Vert y\Vert_2^2.
\end{multline*}
            
\underline{Step 2.} For the infinite product in~\eqref{aux_OU_infprod}, one can again use the expressions~\eqref{K-OU-eigenvalues} to get
\[
\prod_{n=1}^\infty 
\frac{1}{(1-2c\lambda_n)^{d/2}}
=\prod_{n=1}^\infty
\left(
\frac{\beta_n^2+4(a_0^2-c)}{\beta_n^2+4a_0^2}
\right)^{-d/2}
\]
with \(\beta_n\) being the positive solutions of~\eqref{eigenvals_eq_y}. Now let us introduce an auxiliary function
\[
F(z):=2a_0 \sin(zT) + z\cos(zT),
\quad z\in\mathbb{C}.
\]
It can be seen that \(F\) is an entire odd function of order 1, whose roots coincide with those of~\eqref{eigenvals_eq_y}. In particular, \(z=0\) is the root of multiplicity one. Observing that for large \(n\) it holds that \(\beta_n=(n-0.5)\pi/T+O(1/n)\), which implies \(\sum_{n=1}^\infty \beta_n^2<\infty\), one can hence apply the Weierstrass factorisation theorem to decompose
\begin{multline*}
F(z)
=ze^{Cz+D}
\prod_{n\neq 0} \left(
1-\frac{z}{\beta_n}
\right)
e^{z/\beta_n}
=ze^{Cz+D}
\prod_{n=1}^{\infty}
\left(
1-\frac{z}{\beta_n}
\right)
e^{z/\beta_n}
\left(
1+\frac{z}{\beta_n}
\right)
e^{-z/\beta_n}\\
=ze^{Cz+D}
\prod_{n=1}^\infty 
\left(
1-\frac{z^2}{\beta_n^2}
\right)
\end{multline*}
with some \(C,D\in\mathbb{C}\), the second equality following from the symmetry of the roots \(\beta_n\); see, e.g., Chapter 5 Section 2.3 in~\cite{ahlfors1979}. Defining also \(P(z)=z\prod_{n=1}^\infty (1-z^2/\beta_n^2)\) and observing that \(F(z)=e^{Cz+D}P(z)\), both \(F\) and \(P\) being odd, we further conclude that \(C=0\), which yields
\[
F(z)=\operatorname{const} z
\prod_{n=1}^\infty 
\left(
1-\frac{z^2}{\beta_n^2}
\right).
\]
Now one can observe that
\[
\frac{F(2i \sqrt{a_0^2-c})}{F(2 ia_0)}
=\frac{\sqrt{a_0^2-c}}{a_0}
\prod_{n=1}^\infty 
\frac{\beta_n^2+4(a_0^2-c)}{\beta_n^2+4a_0^2}
\]
and, since
\[
F(2ix)
=2a_0\sin(2ixT) + 2ix\cos(2ixT)
=2i(a_0\sinh(2xT)+x\cosh(2xT)),
\]
it follows that
\begin{multline*}
\prod_{n=1}^\infty 
\frac{\beta_n^2+4(a_0^2-c)}{\beta_n^2+4a_0^2}
=\frac{a_0}{\sqrt{a_0^2-c}}
\frac{2i(a_0\sinh(2\sqrt{a_0^2-c}T)+\sqrt{a_0^2-c}\cosh(2\sqrt{a_0^2-c}T))}{2i(a_0\sinh(2a_0 T)+a_0\cosh(2a_0 T))}\\
=\frac{a_0}{\sqrt{a_0^2-c}}
\frac{a_0\sinh(2\sqrt{a_0^2-c}T)+\sqrt{a_0^2-c}\cosh(2\sqrt{a_0^2-c}T)}{a_0 e^{2a_0 T}}\\
=\frac{e^{-2a_0T}}{\sqrt{a_0^2-c}}
\left(
a_0\sinh(2\sqrt{a_0^2-c}T)+\sqrt{a_0^2-c}\cosh(2\sqrt{a_0^2-c}T)
\right).
\end{multline*}
Taking the power \((-d/2)\) yields the claim.
\end{proof}
\begin{lem}\label{eigenvalues_OU}
Consider the operator~\eqref{K-OU-operator} corresponding to the convariance function of the Ornstein-Uhlenbeck process. Its eigenvalues are given by 
    \begin{equation}\label{K-OU-eigenvalues}
    \lambda_n
    := \frac{2}{\beta_n^2+4a_0^2},
    \quad n=1,2,\dots
    \end{equation}
    where \(\beta_n\) are the positive roots of the equation \(\tan \beta T=-\beta/(2a_0)\). The eigenfunctions can be chosen as 
    \begin{equation}\label{phi_y}
        \phi_n(s):=\sin(\beta_n s),
        \quad n\geq 1.
    \end{equation}.
\end{lem}
\begin{proof}
    Let us first define the operator \(\mathcal{D}_s:=\frac{d^2}{ds^2}-4a_0^2\). Since
            \[
            \frac{d^2}{ds^2}e^{-2a_0|t-s|}
            =4a_0^2e^{-2a_0|t-s|}-4a_0\delta(t-s),
            \]
            it can be seen that 
            \[
            (\mathcal{D}_s K^{OU}(\cdot,t))=-2\delta(t-s).
            \]
            Hence, the function \(u\) defined as \(u(s):=(K^{OU} f)(s)\) for some \(f\in L^2([0,T])\) with \(T>0\) fixed satisfies \(\mathcal{D}_s u(s)=-2f(s)\). 
            Noticing that \(\lambda\) (resp.\ \(\phi\neq 0\)) is the eigenvalue (resp.\ eigenfunction) of~\eqref{K-OU-operator} if and only if \(K^{OU} \phi = \lambda \phi\), choosing \(f=\phi\) we get by the above that \(\mathcal{D}_s u(s)=-2\phi(s)=-2u(s)/\lambda\), i.e.,
            \begin{eqnarray}\label{eigenvalues2}
            \frac{d^2}{ds^2} u(s)
            -4a_0^2 u(s)
            &=&-\frac{2}{\lambda}u(s),
            \quad 0<s<T,\\
            u''(s) &=& \left(4a_0^2 - \frac{2}{\lambda}
            \right)
            u(s), 
            \quad 0<s<T,\label{eigenvalues3}
            \end{eqnarray}
            with the boundary conditions \(u(0)=0\) and \(u'(T)=-2a_0u(T)\).
            One easily observes that the non-trivial solution exists if and only if 
            \begin{equation}\label{lambda_bound}
                \frac{2}{\lambda}-4a_0^2=:\beta^2>0,
            \end{equation}
            in which case it is given by 
            \[
            u(s)=A\cos(\beta s)+B\sin(\beta s)
            \]
            with some \(A,B\in\mathbb{R}\). By the boundary conditions we immediately get that \(A=0\), \(B\neq 0\), and it holds that
    \begin{equation}\label{eigenvals_eq_y}
    \beta \cos(\beta T)+2a_0\sin(\beta T)=0
    \quad \Leftrightarrow\quad
    \tan(\beta T)
    =-\frac{\beta}{2a_0}.
    \end{equation}
    Hence, defining \(\beta_n\), \(n\geq 1\), as positive roots of~\eqref{eigenvals_eq_y}, and using again the relation \(\beta^2=2/\lambda-4a_0^2\), we get the claim.
    \end{proof}

    \begin{lem}\label{polynom_moments}
        Let \(Y_t\) be the Ornstein-Uhlenbeck process defined in~\eqref{OU_a0}. Then, for any \(p\geq 1\), it holds that
        \[
        \left(
        \E_{\PP^{OU}_y}\left[
        \Vert Y_t\Vert_2^p
        \right]
        \right)^{1/p}
        \leq 2^{1-1/p}\left(
        e^{-2a_0 t} 
        \Vert y\Vert_2
        +\sqrt{\frac{1-e^{-4a_0 t}}{a_0}}
        \left(
        \frac{\Gamma((d+p)/2)}{\Gamma(d/2)}
        \right)^{1/p}
        \right).
        \]
    \end{lem}
    \begin{proof}
        Similar to Lemma~\ref{OU-expmoments}, the key idea is to decompose \(Y_t\overset{d}{=} \mu^{OU}(t)+Z_t\) under \(\PP_y\), where \(\mu^{OU}(t)=e^{-2a_0 t}y\) and \(Z_t\sim\mathcal{N}(0_d, K^{OU}(t,t)I_d)\) with \(K^{OU}\) defined in~\eqref{OU_cov}. Then it immediately follows that
        \begin{multline*}
            \E_{\PP^{OU}_y}\left[ 
            \Vert Y_t\Vert_2^p
            \right]
            =\E_{\PP^{OU}_y}\left[ 
            \Vert 
            \mu^{OU}(t) + Z_t
            \Vert_2^p
            \right]
            \leq 2^{p-1}\left(
            \Vert \mu^{OU}(t)\Vert_2^p
            +\E_{\PP^{OU}_y} \left[ 
            \Vert Z_t\Vert_2^p
            \right]
            \right)\\
            =2^{p-1}\left(
            \Vert \mu^{OU}(t)\Vert_2^p
            +\left( 
            \frac{1-e^{-4a_0 t}}{2a_0}
            \right)^{p/2}
            \E_{\PP^{OU}_y} \left[ 
            \Vert \xi\Vert_2^p
            \right]
            \right)\\
            =2^{p-1}\left(
            \Vert \mu^{OU}(t)\Vert_2^p
            +\left( 
            \frac{1-e^{-4a_0 t}}{a_0}
            \right)^{p/2}
            \frac{\Gamma((p+d)/2)}{\Gamma(d/2)}
            \right),
        \end{multline*}
        where \(\xi\sim\mathcal{N}(0_d, I_d)\).
        Taking the power \(1/p\) and using the generic inequality \((a+b)^{1/p}\leq a^{1/p} + b^{1/p}\), \(a,b\geq 0\), we get the claim.
    \end{proof}

\bibliographystyle{plain}
\bibliography{biblio}

@book{bogachev2022fokker,
  title={Fokker--Planck--Kolmogorov Equations},
  author={Bogachev, Vladimir I and Krylov, Nicolai V and R{\"o}ckner, Michael and Shaposhnikov, Stanislav V},
  volume={207},
  year={2022},
  publisher={American Mathematical Society}
}

@article{dong2015schauder,
  title={Schauder estimates for higher-order parabolic systems with time irregular coefficients},
  author={Dong, Hongjie and Zhang, Hong},
  journal={Calculus of Variations and Partial Differential Equations},
  volume={54},
  number={1},
  pages={47--74},
  year={2015},
  publisher={Springer}
}

@article{cattiaux2022functional,
  title={Functional inequalities for perturbed measures with applications to log-concave measures and to some Bayesian problems},
  author={Cattiaux, Patrick and Guillin, Arnaud},
  journal={Bernoulli},
  volume={28},
  number={4},
  pages={2294--2321},
  year={2022},
  publisher={Bernoulli Society for Mathematical Statistics and Probability}
}

@book{barbu2024nonlinear,
  title={Nonlinear Fokker-Planck flows and their probabilistic counterparts},
  author={Barbu, Viorel and R{\"o}ckner, Michael},
  volume={2353},
  year={2024},
  publisher={Springer}
}

@book{krylov1996lectures,
  title={Lectures on elliptic and parabolic equations in Holder spaces},
  author={Krylov, Nikola{\u\i} Vladimirovich},
  volume={12},
  year={1996},
  publisher={American Mathematical Soc.}
}

@book{LSU,
  title={Linear and Quasi-linear Equations of Parabolic Type},
  author={Ladyzhenskaia, Ol'ga A. and Solonnikov, Vsevolod A. and Ural'tseva, Nina N.},
  series={American Mathematical Society, translations of mathematical monographs},
  year={1968},
  publisher={American Mathematical Society}
}

@Article{jourdain2008propagation,
 Author = {Jourdain, Benjamin and Malrieu, Florent},
 Title = {Propagation of chaos and {Poincar{\'e}} inequalities for a system of particles interacting through their {CDF}},
 FJournal = {The Annals of Applied Probability},
 Journal = {Ann. Appl. Probab.},
 ISSN = {1050-5164},
 Volume = {18},
 Number = {5},
 Pages = {1706--1736},
 Year = {2008},
 Language = {English},
 DOI = {10.1214/07-AAP513},
 zbMATH = {5374750},
 Zbl = {1185.65013}
}

@book{KS-BM,
  title={Brownian motion and stochastic calculus},
  author={Karatzas, Ioannis and Shreve, Steven},
  year={2014},
  publisher={springer}
}

@book{ahlfors1979,
  title={Complex Analysis},
  subtitle={An introduction to the theory of analytic functions of one complex variable},
  author={Ahlfors, L. V.},
  year={1979},
  publisher={McGraw-Hill Book Company},
  edition={3} 
}

@book{lieberman1996,
  title={Second order parabolic differential equations},
  author={Lieberman, Gary M},
  year={1996},
  publisher={World scientific}
}

@article{BSS24,
  title={The Fokker--Planck--Kolmogorov equation with nonlinear terms of local and nonlocal type},
  author={Vladimir I. Bogachev, Damir I. Salakhov and Stanislav V. Shaposhnikov},
  journal={St. Petersburg Mathematical Journal},
  volume={35},
  number={5},
  pages={749--767},
  year={2024}
}

@book{Barbu2024,
  author    = {Viorel Barbu and Michael R{\"o}ckner},
  title     = {Nonlinear {Fokker-Planck} Flows and Their Probabilistic Counterparts},
  volume    = {2353},
  publisher = {Springer},
  series    = {Lecture Notes in Mathematics},
  year      = {2024}
}

@misc{Bondi2025,
  author        = {Luca Bondi and Elena Issoglio and Francesco Russo},
  title         = {{McKean-Vlasov} Equations with Singular Coefficients---a Review of Recent Results},
  howpublished  = {arXiv preprint arXiv:2507.23553},
  year          = {2025}
}

@article{Cattiaux2022,
  author  = {Patrick Cattiaux and Arnaud Guillin},
  title   = {Functional Inequalities for Perturbed Measures with Applications to Log-Concave Measures and to Some Bayesian Problems},
  journal = {Bernoulli},
  volume  = {28},
  number  = {4},
  pages   = {2294--2321},
  year    = {2022}
}

@article{Grube2024,
  author  = {Sebastian Grube},
  title   = {Strong Solutions to {McKean-Vlasov} {SDEs} with Coefficients of {Nemytskii} Type: The Time-Dependent Case},
  journal = {Journal of Evolution Equations},
  volume  = {24},
  number  = {2},
  pages   = {37},
  year    = {2024}
}

@article{Kumar2022,
  author  = {Chaman Kumar and Neelima and Christoph Reisinger and Wolfgang Stockinger},
  title   = {Well-Posedness and Tamed Schemes for {McKean-Vlasov} Equations with Common Noise},
  journal = {The Annals of Applied Probability},
  volume  = {32},
  number  = {5},
  pages   = {3283--3330},
  year    = {2022}
}

@article{Mishura2020,
  author  = {Yuliya Mishura and Alexander Veretennikov},
  title   = {Existence and Uniqueness Theorems for Solutions of {McKean-Vlasov} Stochastic Equations},
  journal = {Theory of Probability and Mathematical Statistics},
  volume  = {103},
  pages   = {59--101},
  year    = {2020}
}

@misc{Rehmeier2023,
  author        = {Marco Rehmeier},
  title         = {Weighted {$L^{1}$}-Semigroup Approach for Nonlinear {Fokker-Planck} Equations and Generalized {Ornstein-Uhlenbeck} Processes},
  howpublished  = {arXiv preprint arXiv:2308.09420},
  year          = {2023}
}

@article{Ren2023,
  author  = {Panpan Ren},
  title   = {Singular {McKean-Vlasov} {SDEs}: Well-Posedness, Regularities and {Wang's} {Harnack} Inequality},
  journal = {Stochastic Processes and their Applications},
  volume  = {156},
  pages   = {291--311},
  year    = {2023}
}

@article{Rockner2021,
  author  = {Michael R{\"o}ckner and Xicheng Zhang},
  title   = {Well-Posedness of Distribution Dependent {SDEs} with Singular Drifts},
  journal = {Bernoulli},
  volume  = {27},
  number  = {2},
  pages   = {1131--1158},
  year    = {2021}
}

@article{Zhang2005,
  author  = {Xicheng Zhang},
  title   = {Strong Solutions of {SDEs} with Singular Drift and {Sobolev} Diffusion Coefficients},
  journal = {Stochastic Processes and their Applications},
  volume  = {115},
  number  = {11},
  pages   = {1805--1818},
  year    = {2005},
  publisher={Elsevier}
}
\end{document}